\documentclass[12pt, letter]{article}
\usepackage{amsthm}
\usepackage[margin =1in]{geometry}
\usepackage{palatino}
\usepackage{xcolor}
\usepackage[flushleft]{threeparttable}
\usepackage[width=\textwidth]{caption}
\usepackage{subcaption}
\usepackage[pdftex]{graphicx}
\usepackage[skip=0pt]{caption}
\usepackage{setspace}
\usepackage{subcaption}
\usepackage{lscape}
\usepackage{multirow}
\usepackage{multibib}
\usepackage{algorithm}
\usepackage{algpseudocode}
\usepackage{array}
\usepackage[english]{babel}
\newtheorem{theorem}{Theorem}

\newtheorem{definition}{Definition}
\newtheorem{lemma}{Lemma}

\usepackage{titling}
\newcites{SM}{Online Appendix References}
\usepackage[hidelinks]{hyperref}
\hypersetup{
     colorlinks   = false, 
     allcolors    = [RGB]{163,31,52}
}
\usepackage{epigraph}
\usepackage{indentfirst}
\usepackage[round]{natbib}
\usepackage{verbatim}
\usepackage[T1]{fontenc}
\usepackage[utf8x]{inputenc}
\usepackage[pdftex]{graphicx}
\usepackage{hologo}
\usepackage{textcomp}
\usepackage{hanging}
\usepackage{amsmath}
\usepackage{amssymb}
\usepackage{amsfonts}
\usepackage{tikz}
\usepackage{float}
\usepackage{epigraph}
\usepackage{authblk}
\makeatletter
\def\@fnsymbol#1{\ensuremath{\ifcase#1\or \dagger\or \ddagger\or
   \mathsection\or \mathparagraph\or \|\or **\or \dagger\dagger
   \or \ddagger\ddagger \else\@ctrerr\fi}}
    \makeatother

\newcommand{\tx}{\tilde{x}}
\newcommand{\ty}{\tilde{y}}
\newcommand{\tz}{\tilde{z}}
\newcommand{\dist}{\textbf{\text{dist}}}
\newcommand{\mF}{\mathcal{F}}
\newcommand{\pran}[1]{\left(#1\right)}

\begin{document}

\title{\Large Optimizing Scalable Targeted Marketing Policies with Constraints}

\author{Haihao Lu\thanks{Assistant Professor of Operations Management, University of Chicago, Booth School of Business, 5807 S Woodlawn Ave, Chicago, IL 60637, haihao.lu@chicagobooth.edu.} \hspace{1in} Duncan Simester\thanks{NTU Professor of Marketing, Massachusetts Institute of Technology, MIT Sloan School of Management, 100 Main Street, Cambridge, MA 02142, simester@mit.edu.} \hspace{1in} Yuting Zhu\thanks{Assistant Professor of Marketing, National University of Singapore, NUS Business School, 15 Kent Ridge Drive, Singapore, 119245, y.zhu@nus.edu.sg.} }

\footnotetext[1]{Author names are listed alphabetically. The authors received helpful comments from Baris Ata, Daria Dzyabura, Fred Feinberg, Robert M. Freund, Dennis Zhang, Juanjuan Zhang, and Spyros Zoumpoulis; from seminar participants at MIT and National University of Singapore; and from audiences at the 2022 Analytics for X Conference, 2022 Artificial Intelligence in Management Conference, 2022 Conference on Artificial Intelligence, Machine Learning, and Business Analytics, 2022 Marketing Science Conference, 2023 CMIC Conference, 2023 Marketing Analytics Symposium Conference, and 2023 POMS Conference. We would like to acknowledge that cloud resources involved in this research work are partially supported by the NUS Cloud Credits for Research Program.}

\date{December 19, 2023}

\maketitle

\vspace{-0.9cm}
\begin{abstract}
Targeted marketing policies target different customers with different marketing actions. While most research has focused on training targeting policies without managerial constraints, in practice, many firms face managerial constraints when implementing these policies. For example, firms may face volume constraints on the maximum or minimum number of actions they can take, or on the minimum acceptable outcomes for different customer segments. They may also face similarity (fairness) constraints that require similar actions with different groups of customers. Traditional optimization methods face challenges when solving problems with either many customers or many constraints. We show how recent advances in linear programming can be adapted to the targeting of marketing actions. We provide a theoretical guarantee comparing how the proposed algorithm scales compared to state-of-the-art benchmarks (primal simplex, dual simplex and barrier methods). We also extend existing guarantees on optimality and computation speed, by adapting them to accommodate the characteristics of targeting problems. We implement the proposed algorithm using data from a field experiment with over 2 million customers, and six different marketing actions (including a no action ``Control''). We use this application to evaluate the computation speed and range of problems the algorithm can solve, comparing it to benchmark methods. The findings confirm that the algorithm makes it feasible to train large-scale targeting problems that include volume and similarity constraints.

\singlespacing 
\noindent\emph{\textbf{Keywords}}: Targeting, personalization, linear programming, budget constraints, fairness

\end{abstract}

\newpage
\doublespacing

\section{Introduction} \label{sec:intro}
The growth in industry interest in personalization and targeting has been mirrored by rapid growth in academic interest.\footnote{We use the terms ``personalization'' and ``targeting'' interchangeably to describe policies that recommend different marketing actions for different customers.} This academic attention has primarily focused on optimizing targeting policies without constraints. However, in practice, internal business rules or society considerations often impose constraints on firms’ targeting policies. 

Managerial constraints on targeting policies are typically of two types. \textit{Volume} constraints mandate the maximum or minimum number of marketing actions that can be taken. This type of constraint may result from capacity constraints. For example, a firm’s ability to make outbound phone calls may be limited by the availability of trained associates to make these calls. Budget constraints may also impose minimum and/or maximum limits on the total number of marketing actions. For example, Neiman Marcus sends a holiday catalog to a selected sample of its customers each fall. The campaign is largely funded by suppliers, who mandate a minimum number of catalogs that can be sent, while budget constraints also impose a maximum limit. Similarly, the wholesale membership club that provided the data used in this paper, conducts separate spring and fall campaigns in which it sends membership offers to prospective customers. Business rules impose both minimum and maximum limits on the number of prospective customers included in each campaign. 

\textit{Similarity (Fairness)} constraints limit differences in marketing actions taken with different customer segments. These constraints are often motivated by concerns for fairness. For example, a constraint might require that the firm takes similar marketing actions with zip codes that have many versus few African American households, or that customers located near one store are treated with similar marketing actions as customers located near other stores.\footnote{Our focus does not extend to investigating the different definitions of fairness (see the discussion of this issue in Section \ref{sec:review}). We will interpret fairness concerns as firm concerns that similar customers receive similar marketing actions.}


In Table~\ref{table:targetingliterature}, we list eight recent marketing papers published after 2020 that study personalization problems, and summarize the number of customers used in each training dataset. These examples reveal that datasets used to train targeting policies can contain millions of customers. Millions of customers can lead to thousands of customer segments, which collectively can result in millions of decision variables, and large numbers of constraints. For example, volume constraints can quickly become numerous when they are applied to separate customer segments. A credit card firm that wants to offer personalized loan rates to prospective customers, may want to manage both overall credit risk, and credit risk within each tranche of customers (grouped according to credit scores). Similarly, some charities engage in ``laddering'', setting customer-specific guidelines when suggesting donation levels. Fairness constraints also become numerous when used to  compare marketing actions and/or outcomes between individual customer segments. For example, a firm may require that no zip code that contains many households in a protected class receives substantially fewer discounts on average than other zip codes. In Section \ref{sec:problem}, we will show that the number of constraints may grow as a polynomial function of the number of customers and number of segments. 
As a result, methods for designing personalization problems need to be scalable to meet these challenges (\citealt{aipersonalization23}). 


We formulate the problem of training a targeting policy as a linear programming problem with constraints, and illustrate how to incorporate both volume constraints and similarity constraints. The current state-of-the-art methods for solving linear programming (LP) problems are the simplex method (\citealt{bertsimas08, dantzig16}) and the barrier method (\citealt{karmarkar1984new,renegar1988polynomial,monteiro1989interior,wright1997primal}). Both methods are very mature, and are implemented in commercial software, such as Gurobi and CPLEX. They are capable of identifying reliable solutions for medium size problems. However, neither method is well-suited to solving personalization problems that have a large number of decision variables, or a large number of constraints. 


Instead, we adapt and apply an algorithm that leverages the Primal-Dual Hybrid Gradient (PDHG, see \citealt{phdg11}). PDHG methods only require matrix-vector multiplications, which allows these methods to easily scale. In contrast, simplex and barrier methods require solving linear equations using matrix factorization. This leads to two major challenges when solving large-scale problems: (1) while the original targeting problem may be sparse, the matrix factorization can be dense, requiring more memory; (2) it is very challenging (if not impossible) to use modern computing architectures, such as distributed system and GPUs, for matrix factorization. PDHG only requires storing the constraint matrix in memory, and sparse matrix-vector multiplication is easily scaled on modern computing architectures. In light of these fundamental distinctions, \cite{nesterov2013gradient} formally defines optimization methods that require matrix factorization (or linear equation solving) as handling medium-scale problems, and optimization methods that require matrix-vector multiplication as handling large-scale problems.
 
\cite{applegate21a}, \cite{applegate2021infeasibility} and \cite{applegate21b} recently extended PDHG methods to solve linear programming problems. We use a version of their algorithm, which they label ``Primal-Dual Hybrid Gradient for Linear Programming'' (PDLP). They provide theoretical guarantees on computation speed and optimality.  However, the theoretical results only apply to settings with one-sided constraint ($x \geq 0$). As we shall see, targeting problems by their nature include two-sided constraints ($0 \leq x \leq e$). As a result, the theoretical guarantees on performance and optimality established by \cite{applegate21b} no longer apply. We provide the first theoretical guarantee on optimality and performance of the PDLP algorithm for settings with two-sided constraints. We also provide the first documented implementation of a PDHG algorithm in marketing. The application demonstrates that PDHG methods have the potential to make an important contribution to optimizing targeting policies, by expanding the set of problems that can be solved. In particular, they can solve problems with more customers, and more constraints, than state-of-the-art benchmark methods. 

Specifically, we provide two new theoretical results. The first result compares the amount of computation required to solve targeting problems with constraints when using a PDHG algorithm, compared to primal simplex, dual simplex and the barrier method. We prove that one (worst-case) iteration of PDHG requires a lot less computation than the benchmarks. The result is the first formal comparison of the performance of PDHG and these benchmark methods. 
  
The second theoretical result extends the proof in \cite{applegate21b}. They show that the number of steps required to find an $\epsilon$-optimal solution for the PDLP algorithm is on the order of $O(\log(1/\epsilon))$, compared to order $O(1/\epsilon)$ for the standard PDHG algorithm. A theoretical requirement of $O(\log(1/\epsilon))$ iterations versus $O(1/\epsilon)$ iterations can result in substantially faster convergence. Extending this result to accommodate two-sided constraints requires non-trivial extensions of the \cite{applegate21b} result.

In our empirical application, we use PDLP to solve an actual targeting problem. A large United States wholesale membership club wanted to choose which promotions to send to prospective customers. The response functions are estimated using a large-scale field experiment that includes over 2 million customers and six marketing actions (one of which is a ``no action'' control). The experimental variation makes it straightforward to obtain causal estimates of the incremental profit earned from each marketing action. The findings both illustrate a practical implementation of the PDLP algorithm, and provide empirical evidence of how PDHG methods extend the range of solvable targeting problems in the presence of constraints. 


If firms are restricted to using benchmark methods, and yet they still want to satisfy a predetermined set of constraints, they may have to forgo individual personalization. Instead, they can simplify the problem by personalizing marketing actions at the segment level. Replacing individual decision variables with segment-level decision variables reduces the degrees of freedom, which lowers the expected performance of the policies. We can interpret the difference in expected performance between individual-level and segment-level personalization problems as a measure of the value of the proposed method over the benchmark methods. Therefore, we investigate this performance difference, and show that the ability to solve individual-level problems using the proposed method has important implications for firm profits.  

The paper continues in Section~\ref{sec:review}, where we position our contribution with respect to the existing literature. We explicitly model the firm's optimization problem in Section \ref{sec:problem}, and discuss how to incorporate different types of volume and similarity constraints. Section \ref{sec:algorithm} presents details of the algorithm, together with theoretical guarantees. We present an empirical application in Section \ref{sec:validation}, where we apply the algorithm to data from a large-scale field experiment. In Section \ref{sec:manaimplication}, we investigate the economic implications of our findings, by documenting the reduction in expected performance when firms can only personalize policies at the segment level. The paper concludes in Section \ref{sec:conclusion}, where we also highlight promising directions for future research. 

\section{Related Literature}\label{sec:review}
We draw from and contribute to literature on personalization and targeting, fairness, and PDHG optimization methods. We begin by briefly reviewing the literature on personalization and targeting.

\subsection{Personalization and Targeting of Marketing Actions} \label{subsec:lit on targeting}
Research on training targeting policies has grown rapidly in recent years, with most of the research (including this paper) adopting a predict-then-optimize approach. Intuitively, the firm first solves an ``estimation problem'', in which it estimates customer response functions from a sample of training data. After solving the estimation problem, the firm solves an ``optimization problem'', in which the estimated customer response functions are used to design a policy that optimizes an objective function (typically firm profit). The contributions of the two problems was recently investigated by \cite{feldman2022customer}. They compare the relative contribution of a state-of-the-art estimation method with a state-of-the-art optimization method. They demonstrate that sophisticated estimation may not out-perform a relatively simple estimation method that has been properly optimized.

In Table~\ref{table:targetingliterature}, we list eight examples of recent papers that are published after 2020 and investigate methods for targeting a range of different marketing actions. The focus of all of these papers is on the estimation problem, rather than the optimization problem. While some of the papers consider budget constraints, these constraints are simple in nature, and can be solved using rudimentary greedy algorithms. The inclusion of more complicated constraints, and an increase in either the number of decision variables or the number of constraints, require more sophisticated optimization algorithms. We focus on targeting problems with larger and more complicated sets of constraints. 

Personalization and targeting have also received recent interest outside the marketing literature, in both operations research and management. For example, \cite{golrezaei14} consider the problem of personalizing product assortments in a dynamic setting, and propose an index-based multi-armed bandit method. 
\cite{chendavid21} apply statistical learning to customize revenue management policies.  \cite{derakhshan21} also use linear programming to solve personalization problem, where they focus on personalized reserve prices. We refer interested readers to \cite{aipersonalization23} for a recent review of personalization and targeting within and beyond marketing. As far as we know, none of the papers consider the scalability of personalization problems in the presence of constraints. 

Beyond targeting, other research in marketing has investigated how to include constraints when optimizing marketing actions. Examples include research studying conjoint analysis (\citealt{toubia04}), product line design (\citealt{luo10}), retail assortments (\citealt{fisher14}), and content arrangements on social media (\citealt{kanuri18}).

\begin{landscape}
\begin{table}[H] 
\centering
\caption{Constraints in Recent Targeted Marketing Papers.} \vspace{.2in}
\begin{tabular}{ p{5.5cm} p{7.5cm} p{3.5cm} p{3cm} p{3cm}} 
 \hline\hline
 Paper & Problem & Constraint & Optimization Algorithm & \# of Customers (Training)\\
 \hline
 \cite{dmcmc20} & Solicitations for charity  & No Constraints & - & 1,088,310\\
 \\
 \cite{chunprofits20} & Proactive retention campaigns  & Budget Constraint & Greedy & 5,190/2,100\\
 \\
  \cite{prospective20} & Promotions to prospective customers & Budget Constraint & Greedy & 1,185,141 \\
  \\
   \cite{searchtargeting20} & Ranking for query-based search & No Constraints & - & 5,659,229 \\
   \\
   \cite{productassortment21} & Coupons for retailer customers & Budget Constraint & Greedy & 150,094\\
   \\
    \cite{longterm21} & Targeted discounts to retain customers  & No Constraints & - & 45,000 \\
   \\
   \cite{pricetargeting21} & Pricing for a digital firm & No Constraints & - &  7,867\\
   \\
   \cite{hema2023freetrials} & Free trial promotions & No Constraints & - & 337,724\\
 \hline\hline
\end{tabular}
\label{table:targetingliterature}
\end{table}
\end{landscape}

\subsection{Fairness of Marketing Actions} \label{subsec:lit on fairness}
Our findings highlight the challenge of optimizing firm actions in the presence of fairness constraints. There is a long history of studying the role of fairness in firm's marketing actions, and the use of algorithms to make marketing decisions has renewed interest in fairness as a research topic. A distinguishing feature of algorithm fairness is that unfair marketing actions are often an unintended outcome. Algorithms that are designed to optimize seemingly innocuous goals, can lead to unintended and unanticipated differences in how customers are treated. For example, \cite{lambrecht18} documents how an algorithm delivering advertisements promoting job opportunities led to unintended discrimination. Although the advertisement was designed to be gender neutral, the algorithm optimized cost effectiveness, which led to the ad being shown to more men than women. They conclude by showing that this result generalizes across different digital ad platforms. Similarly, \cite{zhang21} shows that the introduction of smart-pricing algorithms increased the gap between the revenue earned by black and white hosts on Airbnb. 

There have been several theoretical studies investigating the implications of fairness for firm's marketing decisions (see for example \citealt{cuiraju07}, \citealt{guo15}, \citealt{lijain16}, \citealt{guojiang16}, and \citealt{fukannan21}). The fairness of firm's pricing decisions has probably received the most attention (see for example  \citealt{wirtz07}, \citealt{campbell07}, \citealt{anderson08}, \citealt{bolton10} and \citealt{allender21}), while research on methods to anticipate or mitigate algorithmic fairness concerns in marketing is only now beginning to emerge. One notable example is \cite{ascarza22}, who propose using bias-eliminating adapted trees to adjust the potential bias in personalization policies. We contribute to this emerging literature by illustrating how to incorporate fairness concerns, and how to optimize these problems when the number of customers or constraints is large. 

Algorithm fairness has also generated considerable interest in the computer science and machine learning literature. This includes an ongoing debate about the definition of fairness. The definitions are many and varied (see for example \citealt{Castelnovo21}, and \citealt{mehrabi21}). It is beyond the scope of this paper to resolve the differences between these definitions. Instead, we will interpret a fairness constraint as an example of a \textit{Similarity} constraint (see the discussion in Section~\ref{sec:problem}). 

The computer science literature has also proposed different methods to identify and restrict discrimination in policies trained using algorithms. The methods can be classified into three types: pre-process, in-process, and post-process (see review papers such as \citealt{pessach21}). Our method belongs to the in-process class of methods, because we explicitly consider fairness as part of the policy training process. 

\subsection{PDHG Optimization Methods} \label{subsec:lit on PDHG}

The algorithm that we propose, leverages the primal-dual hybrid gradient (PDHG, see \citealt{phdg11}). PDHG methods have been widely used in image processing and computer vision applications (e.g., \citealt{zhuchan08}, \citealt{pock09}, \citealt{esser10}, and \citealt{heyuan12}). They are first-order methods, which use gradient information to construct algorithms to find optimal solutions. This class of methods scales very well, and first-order methods have been widely used in many applications, including many machine learning algorithms (\citealt{beck17}).  Recent developments have made the algorithm especially suitable for large-scale linear programming problems (see details in Section \ref{sec:algorithm}).

Within the PDHG class of methods, the algorithm we use falls within the Primal-Dual Hybrid Gradient for Linear Programming (PDLP) family of methods. These methods are very new, with the theoretical foundation described in \cite{applegate21b}. PDLP is a two-loop algorithm. By initiating a second loop, the algorithm reduces the risk that it moves away from the optimal solution before converging. 

\cite{applegate21b} provides theoretical guarantees for both optimality and computation speed. Implementation of the algorithm at scale requires several additional steps, which are described in \cite{applegate21a}. The research teams that authored these papers include highly skilled engineers and researchers at Google, who developed an efficient C++ implementation of the algorithm within the Google OR-Tools suite.\footnote{https://developers.google.com/optimization/lp.} 

As we discuss in the Introduction, we adapt the PDLP algorithm to accommodate the specific characteristics of targeting problems. Our theoretical findings include non-trivial extensions of the guarantees on convergence and optimality to accommodate these characteristics. Moreover, we provide a new theoretical result, comparing the computation requirements of the method with simplex and barrier methods. Finally, we provide the first documented empirical application of this general class of algorithms in the marketing domain.

In the next section, we discuss how to incorporate different types of constraints into personalization problems.

\section{Constraints and Problem Setup}\label{sec:problem}
In this section, we model targeting as a linear programming problem. We begin by describing the setup of the problem. We then discuss the interpretation of the objective function, followed by interpretation of the constraints.

\subsection{Problem Setup} \label{subsec:problem setup}

We assume there are $I$ considered customers and $J$ available marketing actions, where the set of marketing actions could include a "no action" or null treatment.\footnote{Examples of marketing actions include direct mail advertisements offering a "free trial" or a discount.} Each customer belongs to one of $K$ customer segments, where $K \ge 1$ (if $K$ = 1 then all customers belong to the same segment). We assume that customer segments are defined using observable contextual variables (such as gender, race, geographic locations, or past purchasing).

We formulate the firm's problem as follows:
\begin{align}
    \max_{x_i^j} & \ \ \ \sum_{i=1}^I \sum_{j=1}^J p_i^j x_i^j \nonumber \\
    \text{s.t.} &\ \ \ a_k^j \le \sum_{i\in S_k} x_i^j \le b_k^j, \  \text{ for } j=1,...,J, k=1,...,K \ \ \ \textbf{(Volume I)}   \nonumber \\
    & \ \ \  L_k \le \sum_{i\in S_k} \sum_{j=1}^J c^j_i x_i^j \le U_k, \ \text{ for } k=1,...,K \ \ \ \textbf{(Volume II)}  \nonumber \\
    & \ \ \  \frac{1}{n_{k_1}} \sum_{i\in S_{k_1}} x_i^j \le \lambda_j^{k_{1}k_{2}} \frac{1}{n_{k_2}} \sum_{i\in S_{k_2}} x_i^j + g_j^{k_{1}k_{2}}, \nonumber \\
    &\ \ \  \text{ for } j=1,...,J,  k_1=1,...,K, k_2=1,...,K, k_1\neq k_2 \ \ \ \textbf{(Similarity I)} \nonumber \\
     & \ \ \  \frac{1}{n_{k_1}} \sum_{i\in S_{k_1}}\sum_{j=1}^J d_i^j x_i^j \le \gamma^{k_{1}k_{2}} \frac{1}{n_{k_2}} \sum_{i\in S_{k_2}}\sum_{j=1}^J d_i^j x_i^j + h^{k_{1}k_{2}},  \nonumber \\
    &\ \ \  \text{ for } k_1=1,...,K, k_2=1,...,K, k_1\neq k_2  \ \ \ \textbf{(Similarity II)} \nonumber \\
    & \ \ \  \sum_{j=1}^J x_i^j \le M_i,  \ \text{ for } i=1,...,I  \ \ \ \textbf{(Targeting)}\nonumber \\  
    & \ \ \ 0 \le x_i^j \le 1 \ \ \ \textbf{(Feasibility)} \ .
    \label{eq:problem_setup}
\end{align}

The decision variable, $x_i^j\in [0,1]$, represents the probability a given customer $i$ receives marketing action $j$. While $x_i^j$ will normally take the value zero or one, a policy could be stochastic (rather than deterministic).

Problem \eqref{eq:problem_setup} is a linear program, thus it is a convex problem. Furthermore, the targeting and feasibility constraints guarantee that the problem has a bounded and compact feasible region. Therefore, it must have a finite and unique optimal objective value. It is possible to have multiple optimal solutions with the same optimal objective value.

For ease of exposition, we will refer to Problem \eqref{eq:problem_setup} as the ``Individual Personalization with Constraints'' problem, and will use the abbreviation ``IPwC'' throughout the rest of the paper. Notice that in the IPwC problem, the firm chooses marketing actions separately at the individual customer level. In contrast, the volume and similarity constraints are defined at the segment level. 

Defining constraints at the segment level is a natural interpretation for volume constraints, which by their nature, require aggregation across customers. For similarity constraints, the firm could in principle design constraints that compare marketing actions between individual customers. However, this is generally not what we observe in practice. Firms face a practical challenge when using similarity constraints to protect a disadvantaged class: the firm may not know which individual customers fall into the disadvantaged class. Few firms document individual characteristics such as race, color, religious creed, national origin, sexual orientation, or disability, which are often used to define protected classes under federal and state laws. In the absence of individual level data, firms can use zip code-level census information to design similarity constraints. For example, suppose a firm wants to impose the similarity constraint that African American households receive a similar number of discounts as other households. Rather than designing a constraint at the individual household level, it can require that zip codes with many African American households receive a similar proportion of discounts as other zip codes. Defining similarity constraints at the segment-level provides a natural solution when customer-level data is unavailable. 

We make two further remarks about choosing marketing actions at the individual customer level, while constraints are defined at the segment level. First, if we define each customer as a separate segment, this is equivalent to defining constraints at the individual customer level. The problem generalizes to include this special case. Second, in Section \ref{sec:manaimplication}, we consider an alternative formulation in which marketing actions are chosen at the segment level, instead of the individual-level (we label this the ``SPwC'' problem, or ``Segment Personalization with Constraints''). This is a simpler problem, and we will show that it can be solved by existing methods. We will interpret the difference in the profit earned from the IPwC and SPwC problems as a measure of the value contributed by the proposed algorithm. 

\subsection{Interpretation of the Objective Function} \label{subsec:obj function}
The firm's objective in~\eqref{eq:problem_setup} is to maximize the incremental profit it earns across all customers and all marketing actions. We denote the incremental profit that the firm earns from customer $i$ if it receives marketing action $j$ as $p_i^j$. Estimating the incremental response to marketing actions ($p_i^j$) has been the primary goal of many marketing studies, including many of the recent personalization papers listed in Table~\ref{table:targetingliterature}. We focus on how the firm uses these estimated incremental responses to identify an optimal personalization policy.\footnote{We discuss how we predict $p_i^j$ in our empirical application in Section \ref{sec:estimation}. We do note that we stay in the predict-then-optimize paradigm; the firm first predicts $p_i^j$, and then uses the predicted outcomes to generate optimized decision variables. It is interesting, but beyond the scope of our paper, to explore how to reconcile the misaligned objectives in prediction and optimization (e.g., \citealt{spo22}).} 

The incremental profit is calculated as the difference between: (a) the profit earned from customer $i$ if the customer receives marketing action $j$, and (b) the profit earned from customer $i$ if no action is taken. Alternatively, we could treat the null action as a separate action in itself, but we would then adjust how we measure outcomes. If we treated the null action as an action in itself, we would measure profits as the profit earned from each action, rather than the difference in profit compared to the null action.   

Without changing the model, we can alternatively redefine $p_i^j$ to measure increments of revenue, units sold, or some other managerially relevant observable outcome. We interpret a no action or null marketing action as a decision not to implement any of the $J$ actions ("business as usual"). The same logic applies to all constraints. To simplify the exposition, we will drop the "incremental" element in our explanation of the constraints.

\subsection{Interpretation of the Constraints} \label{subsec:constraints}
The first set of constraints $a_k^j \le \sum_{i\in S_k} x_i^j \le b_k^j$ (Volume I) represent volume constraints on each marketing action. Here, $S_k$ is the set of customers in segment $k$, $a_k^j$ is the lower bound for the total number of customers given marketing action $j$ in segment $k$, and $b_k^j$ is the upper bound for the total number of customers given marketing action $j$ in customer segment $k$. Volume constraints can include both a minimum or maximum requirement. For example, we earlier cited the example of the Neiman Marcus holiday catalog, where budget constraints impose an upper limit on how many customers receive the catalog, and agreements with suppliers (who fund the catalog) impose a minimum number of recipients. Notice that the framework also allows the minimum and maximum volume constraints to vary by customer segment ($k$) and marketing action ($j$); $a_k^j$ and $b_k^j$ can vary across $k$ and $j$. 

The second set of constraints $L_k \le \sum_{i\in S_k} \sum_{j=1}^J c^j_i x_i^j \le U_k$ (Volume II) are volume constraints across all marketing actions. Here, $L_k$ denotes the lower bound for a combination of all marketing actions in customer segment $k$, and $U_k$ denotes the upper bound for a combination of all marketing actions in customer segment $k$. The combination of the marketing actions is determined by parameter $c_i^j$. We allow the combination to differ across different customers. This is consistent with charities ``laddering'' their proposed donations, by setting customer-specific donation guidelines. We can also re-interpret this second set of constraints as constraints on marketing outcomes (e.g. the number of customers that respond), rather than constraints on marketing actions. The formulation of the constraints are unchanged under this alternative interpretation.

Volume constraints include several common constraints discussed in the literature. For example, a budget constraint often takes this form. In a budget constraint, $c_i^j$ denotes the costs for marketing action $j$ to customer $i$, $L_k$ is zero, and $U_k$ denotes the total budget number for customer segment $k$. It is possible that $c_i^j$ takes the same value across different customers $i$. Performance constraints may also take this form. Performance constraints impose requirements on a measurable outcome of the targeting policy. For example, although a firm’s objective may be to maximize profits (including the cost of the marketing actions), a manager's goals might include a requirement that this year's revenue is no lower than last year (\citealt{oyer98}). In this case, $c_i^j$ denotes the revenue if we give the customer $i$ marketing action $j$, $L_k$ is the performance requirement number for customer segment $k$, and $U_k$ is equal to infinity. 

The third set of constraints $ \sum_{i\in S_{k_1}} x_i^j/n_{k_1} \le \lambda_j^{k_{1}k_{2}}  \sum_{i\in S_{k_2}} x_i^j/n_{k_2} + g_j^{k_{1}k_{2}}$ (Similarity I) model the similarity constraints for each marketing action. These similarity constraints restrict the difference in the actions taken with different customer segments, and will often be motivated by concerns about fairness (see for example, \citealt{Castelnovo21}, \citealt{mehrabi21}). The total number of customers in segment $k$ is denoted by $n_k$, and $\lambda_j^{k_{1}k_{2}}$ and $g_j^{k_{1}k_{2}}$ restrict the difference between customer segments $k_{1}$ and $k_{2}$. Our framework can capture both subtraction and division differences. When $g_j^{k_{1}k_{2}}=0$, the constraint specifies that the division difference in the proportion of customers receiving a given marketing action between two customer segments cannot be larger than $\lambda_j^{k_{1}k_{2}}$. When $\lambda_j^{k_{1}k_{2}} = 1$, the constraint specifies that the subtraction difference in the proportion of customers receiving a given marketing action between two customer segments cannot be larger than $g_j^{k_{1}k_{2}}$. 


Recall our earlier example, in which a firm wants to impose a similarity constraint that African American households receive a similar number of discounts as non-African American households. We recognized that defining segments using zip codes, and then using census data to characterize the segments, provides a natural solution when customer-level data is unavailable. For illustration, assume a firm uses census data to identify zip codes with at least $X\%$ African American households. We can treat households in these zip codes as segment $k_1$, and the remaining households as segment $k_2$. The firm might require that the proportion of households that receive a discount cannot vary by more than 5\% between the two segments. We can then formulate this constraint as $(\sum_{i\in S_{k_1}} x_i^j/n_{k_1})/(\sum_{i\in S_{k_2}} x_i^j/n_{k_2}) \le 1.05$ and $(\sum_{i\in S_{k_2}} x_i^j/n_{k_2}) / (\sum_{i\in S_{k_1}} x_i^j/n_{k_1}) \le 1.05$. Alternatively, we could define each zip code as a separate segment, and require that the proportion of households receiving discounts in zip codes with many African Americans does not vary by more than 5\% from any zip code with few African Americans. 

We can easily include restrictions in which there is a minimum (instead of maximum) difference in the proportion of customers receiving a given marketing action. The setup also allows a firm to restrict the difference in the \textit{number} of customers who receive a marketing action in different marketing segments (instead of the \textit{proportion}).\footnote{For example, we might require that at least twice as many disadvantaged customers receive discounts as customers who are not disadvantaged. In this case, we can treat customers who are not disadvantaged as customer segment $k_1$, and disadvantaged customers as customer segment $k_2$. The formulation of this example is $\sum_{i\in S_{k_1}} x_i^j/\sum_{i\in S_{k_2}} x_i^j \le 1/2$ with $\lambda_j^{k_{1}k_{2}} = n_{k_1}/2n_{k_2}$ and $g_j^{k_{1}k_{2}}=0$.}

The fourth set of constraints $ \sum_{i\in S_{k_1}}\sum_{j=1}^J d_i^j x_i^j/n_{k_{1}} \le \gamma^{k_{1}k_{2}}  \sum_{i\in S_{k_2}}\sum_{j=1}^J d_i^j x_i^j/n_{k_{2}}+ h^{k_{1}k_{2}}$ (Similarity II) impose similarity constraints on all marketing actions (or marketing outcomes). The difference between customer segments $k_{1}$ and $k_{2}$ is restricted by $\gamma^{k_{1}k_{2}}$ and $h^{k_{1}k_{2}}$, and $d_i^j$ is the weighting factor to determine the combination of all marketing actions. For example, the firm might want to require that the budget allocated to disadvantaged customers is at least twice the budget allocated to customers who are not disadvantaged. In this example, we can again treat customers who are not disadvantaged as customer segment $k_1$, disadvantaged customers as customer segment $k_2$, and $d_i^j$ as the costs for marketing action $j$ to customer $i$ to derive the formulation. 

The last two constraints $\sum_{j=1}^J x_i^j \le M_i$ (Targeting) and $0\le x_i^j \le 1$ (Feasibility) restrict the firm's action space and represent the key characteristics of a personalization problem. Here, $M_i \in (0,J]$ denotes the maximum marketing actions that can be taken for customer $i$. We take a generous perspective, by recognizing that firms may want to send multiple marketing actions to the same customer.\footnote{When $M_i>1$, our framework implicitly assumes that the profit that the firm earns from customer $i$ if it receives marketing action $j$ is independent across different $j$, which might not hold in reality. In Appendix \ref{subsec:setup_another}, we show that a personalization problem with constraints can still be modeled as a linear programming problem even if we remove the independence assumption across different marketing actions. The algorithms and theoretical guarantees provided still hold.} Notice that marketing actions can cause negative treatment effects ($p_i^j$ < 0), but it is not possible to take the negative of a marketing action. As we will show in Section \ref{sec:algorithm}, this restriction on the firm's action space plays a key role in our adaptation of existing theoretical guarantees to personalization problems with constraints.   


\subsection{Summary} \label{subsec:constraints_summary}
With the IPwC problem formulated in Problem (\ref{eq:problem_setup}), the number of constraints is potentially large. However, in practice, some but not all of the constraints may be relevant. By specifying several different types of constraints, we aim to provide a menu of options for firms (and researchers) wanting to incorporate constraints into personalization policies. 

In Table \ref{table:constraints_summary}, we provide an example of each type of constraint. We also include a calculation of the number of individual constraints required for a single set of each type of constraint, where the set encompasses every combination of customers and segments (where applicable). In our application in Section~\ref{sec:validation}, we use a total of $J = 5$ marketing actions, $I=2,065,758$ customers, and up to $K=229$ segments (zip codes). In Table \ref{table:constraints_summary}, we calculate the number of individual constraints required both in general terms, and for this application.

\begin{table}[h]
\doublespacing
\footnotesize
\centering
\begin{threeparttable}
\caption{Number of Individual Constraints Required for Each Type of Constraint}\vspace{.1in}
\begin{tabular}{ p{3cm} p{7cm} c c} 
 \hline\hline
\textbf{Type of Constraint} & \textbf{Example} & \multicolumn{2}{c}{\textbf{Number of Constraints}} \\
& & \textbf{General Formula} & \textbf{Our Application}\\
  \hline
  Volume I & At least 10\% and no more than 30\% of customers receive a free trial promotion in each segment. & 2KJ & 2,290 \\
  Volume II & Minimum and maximum number of free trial and discount promotions received on average by households in each segment. & 2K & 458 \\
 Similarity I  &  The proportion of households receiving a discount promotion does not vary by more than 50\% across all of the segments. 	 & 	K(K-1)J & 261,060  \\ 
 Similarity II & The average number of promotions received by customers in  a segment does not vary by more than 50\% across the segments.  & K(K-1) & 52,212 \\
 Targeting & Customers can at most receive three promotions. & I & 2,065,758\\
  \textbf{Total} & - & - & \textbf{2,381,778}\\
 \hline\hline
\label{table:constraints_summary}
\end{tabular}
\begin{tablenotes}[para,flushleft]
\item \small \singlespacing \vspace{-.6in} Notes. The table reports the number of individual constraints required for each type of constraint, where we assume the individual constraints encompass every combination of customers and segments (where applicable). We report both the general formula and the specific number of constraints required for our empirical application (described in Section~\ref{sec:validation}). 
\end{tablenotes}
\end{threeparttable}
\end{table}

A complete set of constraints requires a total of $2,381,778$ individual constraints ($2JK + 2K + K(K-1)J + K(K-1) + I$), comprised of: $2,290$ Volume I constraints ($2JK$), $458$ Volume II constraints ($2K$), $261,060$ Similarity I constraints ($K(K-1)J$), $52,212$ Similarity II constraints ($K(K-1)$), and $2,065,758$ Targeting constraints ($I$). However, even this is not an upper bound on the number of constraints, as the firm might choose to use multiple versions of a constraint (e.g. Similarity II constraints on both race and gender). In our empirical application (Section \ref{sec:validation}), we show that the number of constraints can pose a challenge for benchmark methods, while our proposed algorithm scales well. We also investigate which types of constraints are particularly challenging. 

Before describing these findings, we first formally present the optimization algorithm, and provide guarantees on feasibility, performance and computation speed.

\section{Algorithm and Theoretical Guarantees}\label{sec:algorithm}

In this section, we present the central ideas in our proposed algorithm, and then provide theoretical guarantees on the performance of the algorithm when applied to IPwC problems.

The method we utilize is PDLP~\citep{applegate21a}, a new first-order method for solving linear programming problems. As we discussed earlier, the key reason for the scalability of our method is that the computational bottleneck of PDLP is matrix-vector multiplication, which can take advantage of modern computing architectures, such as GPUs and distributed computing, and only requires storing the constraint matrix in memory. On the other hand, state-of-the-art benchmark methods, such as the primal simplex, dual simplex and barrier methods, require solving linear equations using matrix factorization. Next, we will describe the key operations in the algorithm to illustrate how only matrix-vector multiplication is required. 


Suppose $W$ represents $I\times J$. We reorganize Problem \eqref{eq:problem_setup} to a general version:
\begin{align}
    \min_{x\in R^{W}} & \ \ \ p^T x \nonumber \\
     \text{s.t.} & \ Gx \le h \nonumber \\
     & 0 \le x \le e, 
     \label{eq:problem_primal}
\end{align}
where we stack all decision variables $x_{i}^{j}$ into a vector $x \in R^{W}$, all negative incremental profit $-p_i^j$ into a vector $p \in R^{W}$, and use $e$ to represent the ``all-one'' vector in $R^{W}$. All of our constraints are inequality constraints, and we can organize them into matrix form $Gx\le h$, where matrix $G \in R^{L \times W}$ and $h \in R^{L}$. Here, $L$ indicates the number of constraints. 

By dualizing the constraint $G x \le h$, we obtain the primal-dual form of the problem:
\begin{equation}\label{eq:problem_primal_dual}
   \min_{x\in X} \max_{y\in Y} \ p^T x - h^Ty + y^T Gx,
\end{equation}
where $X=\{x\in R^{W}: 0\le x \le e\}$ and $Y = \{y\in R^{L}: y \ge 0\}$. Furthermore, by maximizing the primal variable $x$ over set $X$, we obtain the dual problem:
\begin{align}
    \max_{y\ge 0} \ & \sum_{w=1}^{W} (p_w+G_w^T y)^+ - h^T y ,
     \label{eq:problem_dual}
\end{align}
where $^+$ refers to the positive part of a scalar.
By duality theory, we know that the optimal solution to the primal-dual problem \eqref{eq:problem_primal_dual} gives us an optimal solution to the primal problem \eqref{eq:problem_primal} and an optimal solution to the dual problem \eqref{eq:problem_dual}. In order to avoid projection onto the potentially complicated polytope constraint, we study how to solve the primal-dual formulation~\eqref{eq:problem_primal_dual}. 

To solve problem \eqref{eq:problem_primal_dual}, the base algorithm we utilize is the Primal-Dual Hybrid Gradient (PDHG) (\citealt{phdg11}). PDHG is an iterated method. It initializes with a primal-dual solution pair, and keep updating this primal-dual solution pair with the following rule until a high-quality primal-dual solution pair is obtained:
\begin{align}
     x^{new} &= \text{proj}_{X} (x^{old} - \eta p - \eta G^T y^{old}) \nonumber\\
     y^{new} &= \text{proj}_{Y} (y^{old}-\tau h + \tau G(2x^{new}-x^{old})) \ ,
      \label{eq:pdhg_gradient}
\end{align}
where $x^{old}$ and $y^{old}$ are the previous primal and dual solution respectively, $x^{new}$ and $y^{new}$ are the next primal and dual solution respectively. In Equation (\ref{eq:pdhg_gradient}),
$\text{proj}_{X}(\cdot)$ denotes the projection from $(\cdot)$ to the set $X$, and $\text{proj}_{Y}(\cdot)$ denotes the projection from $(\cdot)$ to the set $Y$. $\eta, \tau>0$ are two parameters of the algorithm, which are called primal step-size and dual step-size respectively\footnote{PDLP selects the step-size $\eta$ and $\tau$ adaptively (see Section 3.1 in \citealt{applegate21a} for more details).}. As we can see from the update rule (\ref{eq:pdhg_gradient}), the iterations are "matrix-free" in the sense that we only require matrix-vector multiplications of the data matrix, without the need to solve any linear equations. 


We utilize PDLP, which is a variant of PDHG that is designed to solve linear programming problems (\citealt{applegate21a}, \citealt{applegate2021infeasibility} and \citealt{applegate21b}). The algorithm, which is formally defined in Algorithm \ref{alg:twoloop}, is a two-loop algorithm. We use $n$ to denote the outer loop counter and $t$ to denote the inner loop counter. We initialize the algorithm from an arbitrary primal-dual solution pair $(x^{0,0}, y^{0,0})$. Suppose we are now at the $n$-th outer loop. In the $t$-th inner loop, 
we run one iteration of PDHG update to obtain the next step, i.e., $(x^{n,t+1},y^{n,t+1}) \gets PDHG((x^{n,t},y^{n,t}))$. More formally, it refers to obtaining a new primal-dual solution pair $(x^{n,t+1},y^{n,t+1})$ from the old primal-dual solution pair $(x^{n,t},y^{n,t})$ using the PDHG update rule~\eqref{eq:pdhg_gradient}.  We then compute the average of the primal and the dual sequence in the $n$-th outer iteration, namely $\bar{x}^{n,t+1} \gets \sum^{t+1}_{i=1}x^{n,i}/(t+1)$, $\bar{y}^{n,t+1} \gets \sum^{t+1}_{i=1}y^{n,i}/(t+1)$. We keep running the inner iterations until the normalized duality gap (formally defined in Appendix \ref{eq:rhorz-pos}) contracts by a constant factor. We then move to the next outer loop, ($n \gets n+1$), where the initial solution is the average solution in the previous outer loop: $(\bar{x}^{n+1,0}, \bar{y}^{n+1,0}) \gets (\bar{x}^{n,t}, \bar{y}^{n,t})$.

 
\medskip
\begin{algorithm}
\caption{The PDLP Algorithm}
\label{alg:twoloop}
\begin{algorithmic}
\State \textbf{Input}: Initialize the outer loop: $n \gets 0$; Initialize a primal-dual solution pair $(x^{0,0},y^{0,0})$ ;
\While{}
\State Initialize the inner loop: $t \gets 0$;
\While{}
\State $(x^{n,t+1},y^{n,t+1}) \gets PDHG((x^{n,t},y^{n,t}))$;
\State $\bar{x}^{n,t+1} \gets \sum^{t+1}_{i=1}x^{n,i}/(t+1)$, $\bar{y}^{n,t+1} \gets \sum^{t+1}_{i=1}y^{n,i}/(t+1)$;
\State $t \gets t+1$;
\EndWhile \ normalized duality gap decay condition holds;
\State $\bar{x}^{n+1,0} \gets \bar{x}^{n,t}, \bar{y}^{n+1,0} \gets \bar{y}^{n,t}$;
\State $n \gets n+1$;
\EndWhile \ $(x^{n,0},y^{n,0})$ satisfies the termination criteria\footnotemark.
\end{algorithmic}
\end{algorithm}



\footnotetext{PDLP terminates with an approximately optimal solution, which has a small relative KKT error (see Section 4.1 in \citealt{applegate21a} for additional details).}
In Figure \ref{fig:intuition}, we offer a simple example $\min_x \max_y xy$ to illustrate why two-loop iterations can help with convergence. The grey line shows the convergence path for one-loop iteration using Equation \eqref{eq:pdhg_gradient}, and the black line shows the convergence path for two-loop iterations documented in Algorithm \ref{alg:twoloop}.\footnote{To be specific, each point of both grey and black lines is the average of all past iterates. Notice that when our algorithm stops the inner loop and starts a new outer loop, we take the average for the new outer loop. We show the average of past iterates to show that the key difference between the two algorithms is from the restart scheme.} The starting point is $(5,5)$ and the optimal solution is $(0,0)$.

\begin{figure}[h]
\begin{center}
\caption{Iterates of the one-loop and the two-loop PDHG}
\includegraphics[width=0.7\textwidth]{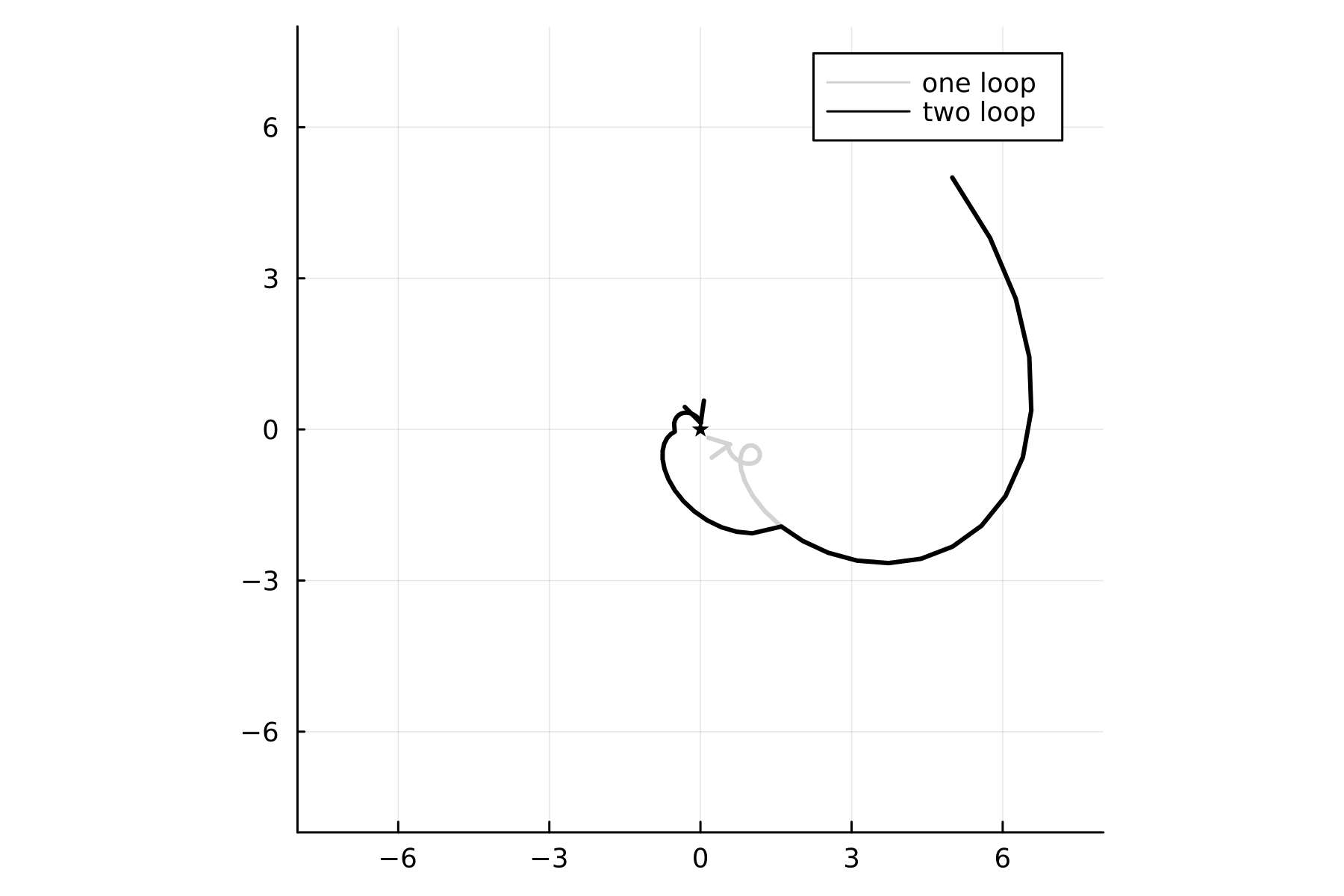}
\label{fig:intuition}
\end{center}
\centerline{\begin{minipage}{.75\textwidth}
\vspace{0.1cm} 
\small  Notes: The figure illustrates the iterates of the one-loop and two-loop PDHG for solving an illustrative problem: $\min_x \max_y xy$ with starting point $(5,5)$ and optimal solution $(0,0)$. 
\end{minipage}}
\end{figure}

The iterates of one loop algorithm (PDHG) spiral in and converge to the optimal solution (the grey line). While the one-loop algorithm can eventually converge to the optimal solution, the performance of the solution may fluctuate.  \cite{phdg16} formally prove that the average iterate converges to the optimal solution with complexity $O(1/\epsilon)$. This implies that the one-loop algorithm requires $O(1/\epsilon)$ steps to find an $\epsilon$-optimal solution. 

Intuitively, the average iterate is close to the optimal solution when one spiral finishes, due to the spiral-in structure, and then the average iterate may move away from the center of the spiral (the optimal solution). The two loop algorithm (PDLP) aims to restart once a spiral finishes (the black line). At the restarting time (the discontinuous points on the black line), the average iterate is close to the optimal solution. The algorithm restarts from this average solution, avoiding moving far-away from the optimal solution as in the one-loop algorithm. In Theorem \ref{thm:twoloop}, we formally prove that for the general LP of Form \eqref{eq:problem_primal}, the two loop algorithm has $O(\log(1/\epsilon))$ complexity.


We next present two theoretical results.
Theorem \ref{thm:largescale} states that one iteration of the proposed algorithm requires less computation when solving larger constrained personalization problems of Form \eqref{eq:problem_setup} than the simplex and barrier methods. Theorem \ref{thm:twoloop} presents the number of iterations of Algorithm \ref{alg:twoloop} that is needed to find an $\varepsilon$-close solution to \eqref{eq:problem_primal_dual}. These two results provide a theoretical foundation for using Algorithm \ref{alg:twoloop} to solve IPwC problems. 

In Theorem \ref{thm:largescale}, we use the notation ``$nnz$'' to describe the number of non-zeros in the constraint matrix $G$.

\begin{theorem}
    
\label{thm:largescale}

In the worst case scenario Algorithm \ref{alg:twoloop} requires $O(nnz)$ floating point operations per iteration, while one major iteration in the simplex and barrier methods requires $O(\min\{L^3, W^3\})$ floating point operations. 
\end{theorem}


\onehalfspacing
\begin{proof}
\setlength{\leftskip}{0.5cm}
In algorithm 1, the major cost per iteration is the two matrix-vector multiplications in the PDHG update in Equation \eqref{eq:pdhg_gradient}, and each matrix-vector multiplication requires $nnz$ floating point operations and $nnz \le L*W$.

\medskip
\noindent The typical implementation of the simplex method utilizes the revised simplex method. The computational bottleneck of the revised simplex method is that the linear equation solves on a basis matrix in each major iteration. Since the basis matrix is not symmetric, the linear equation solving is usually by a (sparse) LU decomposition, which requires $O(L^3)$ or $O(W^3)$ floating point operations in the worst scenario, depending on whether we use primal simplex or dual simplex.

\medskip
\noindent Every barrier method iteration requires solving a symmetric linear equation, which is the computational bottleneck of the barrier method.  This step is usually performed by a (sparse) Cholesky decomposition, which requires $O((L+W)^3)$ floating point operations in the worst scenario.


\end{proof}
\setlength{\leftskip}{0cm}
\doublespacing

Theorem \ref{thm:largescale} proves that for large-scale IPwC problems, a step in Algorithm \ref{alg:twoloop} is typically much cheaper than a step of the barrier method, or a major step in the simplex method, by noticing $O(nnz)\ll O(\min\{L^3, W^3\})$. Here, ``large-scale" can refer to the number of customers ($I$), the number of customer segments ($K$) and/or the number of marketing actions ($J$). The only condition we require is that the number of customers is  large relative to the number of segments and the number of marketing actions, which will essentially always be true for IPwC problems (see our empirical application in Section \ref{sec:validation}). 

In Theorem \ref{thm:twoloop}, we show that Algorithm \ref{alg:twoloop} can achieve global linear convergence for IPwC problems. 

\begin{theorem}\label{thm:twoloop}
 Algorithm \ref{alg:twoloop} can achieve global linear convergence of the problem in Equation (\ref{eq:problem_setup}). Specifically, Algorithm \ref{alg:twoloop} requires at most $O(\log(\frac{1}{\varepsilon}))$ number of iterations to find an $\varepsilon$-approximate solution\footnote{Similar to the barrier method, PDHG and PDLP asymptotically find an optimal solution. This is different from the simplex method, which finds an optimal solution in finite (though perhaps exponentially many) iterations.} to \eqref{eq:problem_primal_dual} such that the distance between this solution to an optimal solution is at most $\varepsilon$.
\end{theorem}

\cite{phdg16} establish that the standard one loop PDHG algorithm requires an order of $O(\frac{1}{\varepsilon})$ iterations to find an $\epsilon$-optimal solution. More recently, \cite{applegate21b} prove that the order of iterations required for the two-loop PDLP algorithm is only $O(\log(\frac{1}{\varepsilon}))$. However, the \cite{applegate21b} results only apply to linear programming problems with one-sided constraints, while the Feasibility constraints in Problem \ref{eq:problem_setup} are two-sided constraints: $0 \le x_i^j \le 1$.\footnote{Or equivalently, $0\le x \le e$ in Problem \eqref{eq:problem_primal}.}  Therefore, the \cite{applegate21b} results do not apply. In Theorem \ref{thm:twoloop}, we extend the \cite{applegate21b} results to the case of two-sided constraints. In particular, we show that PDLP with two-sided constraints converges to an $\epsilon$-optimal solution in at most $O(\log(\frac{1}{\varepsilon}))$ iterations. The difference between a theoretical requirement of $O(\log(\frac{1}{\varepsilon}))$ and $O(\frac{1}{\varepsilon})$ iterations indicates a substantial improvement in computation speed. 

In principle, we could treat $x_i^j\le 1$ as a linear constraint, instead of a variable bound. We could then use the machinery developed in \cite{applegate21b} to obtain a $O(\log(\frac{1}{\varepsilon}))$ theoretical rate. However, this is not what is implemented in PDLP. Moreover, treating $x_i^j\le 1$ as a linear constraint would introduce redundant dual variables (associated with these constraints), which could significantly slow down the convergence of the algorithm.

While the logic in the proof of Theorem \ref{thm:twoloop} still follows from \cite{applegate21b}, the main difficulty is to show that the corresponding primal-dual formulation \eqref{eq:problem_primal_dual} is sharp when the problem includes two-sided constraints (see definition in Appendix \ref{subsec:largescale_proof}). It is very challenging (if not impossible) to extend their analysis based on Hoffman constant into our setting, because the dual problem of \eqref{eq:problem_primal_dual} is an unconstrained minimization problem with a piecewise-linear objective function, and the KKT system studied in \cite{applegate21b} is no longer valid. We here utilize a different proof idea based on studying the box constraint in \eqref{eq:problem_primal} and the piecewise linear structure in \eqref{eq:problem_dual} to show the sharpness result for \eqref{eq:problem_primal_dual}. More details of the proof is presented in Appendix \ref{subsec:largescale_proof}.

Once the sharpness condition is established, we can then prove Theorem \ref{thm:twoloop} following~\cite{applegate21b}. At a high level, one can show that under the sharpness condition the iterates in Equation \eqref{eq:pdhg_gradient} have sublinear convergence rate, i.e.,
$$
\textit{metric}(\bar{z}^{n,t})\le \frac{C}{t} \textit{metric}(\bar{z}^{n,0}) \ ,
$$
where we denote $z=(x,y)$ as the combined primal-dual variable, $C$ is a problem-dependent constant (see details in Appendix \ref{subsec:largescale_proof}), and $\textit{metric}$ is a non-negative metric to measure the quality of the solution (formally defined in Equation \eqref{eq:rhorz-pos}). The metric will be zero if the algorithm reaches an optimal solution. Recall that $n$ represents the outer loop iteration number, and $t$ is the inter loop iteration number. Detailed expressions for $\bar{x}^{n,t}$ and $\bar{y}^{n,t}$ can be found in Algorithm \ref{alg:twoloop}. 

Generally speaking, the restart condition for the outer loop can guarantee that if $t\ge 2C$, we have $$\textit{metric}(\bar{z}^{n+1,0})=\textit{metric}(\bar{z}^{n,t})\le \frac{1}{2} \textit{metric}(\bar{z}^{n,0})\ ,$$
thus the metric halves after one outer iteration. This guarantees the global linear convergence of Algorithm \ref{alg:twoloop}. The formal proof of this theorem is provided in Appendix \ref{subsec:largescale_proof}.


In the next section, we use data from a large-scale field experiment to present an application, in which we use PDLP to solve an IPwC problem. We compare the performance of the algorithm with state-of-the-art benchmark methods. 

\section{Empirical Application}\label{sec:validation}
We begin this section by introducing the data and business problem that we address in this application. We then explain how we design customer segments, and how we estimate customer response functions for each marketing action. We introduce benchmark methods, hardware resources and performance measures, and then present a series of results comparing the performance of the algorithm with the benchmark methods.   

\subsection{Data and Business Problem}\label{sec:estimation}
The data was provided by a large American retailer. The retailer operates membership wholesale club stores selling a broad range of products, including electronics, furniture, outdoor, toys, jewelry, clothing, and grocery items. The retailer uses promotions to attract new members, and has implemented large-scale experiments to help it personalize which promotions it should send to different prospective customers. Customers must register for club membership in order to purchase, and the retailer matches the name and address provided at registration to track which customers responded to each promotional offer. 

The data describes a large field experiment conducted by the firm in 2015.\footnote{Data from this experiment has been used in several previous studies, including: \cite{simester2020efficiently} and \cite{prospective20}.} The experiment's goal was to compare how prospective customers responded to five different direct mail promotions, and a no-action control (a total of six experimental conditions). The firm wanted to use this information to design a targeting policy that recommends which marketing action to choose for each customer. A customer represents a separate prospective household, and the customers were randomly assigned to the six experimental conditions. In total, the experiment included approximately 2.4 million unique customers (households).

The profit earned from each customer was measured over the twelve months after the date the promotions were mailed. The profit measure included mailing costs, membership revenue, and profits earned from purchases in the store (if any). The different direct mail promotions included in the experiment were expected to impact these profit components in different ways. For example, some of the promotions offered free trial memberships of different lengths. Longer trials tend to increase adoption, but yield less membership revenue, and attract customers who spend less in the stores. The experimental conditions also included a discounted membership offer. Compared to free trials, discounted memberships tend to convert fewer prospective customers into members, but they generate more membership revenue, and tend to attract customers who spend more in the stores.   

Low response rates are particularly common when prospecting for new customers. As a result, only a relatively small number of customers responded in each experimental condition, and so across the 2.4 million customers, the profit was negative for most customers in the five promotion conditions (due to mailing costs), and zero for most customers in the no-action control (there were no mailing costs in the no-action control condition). However, for the customers who did respond, the twelve-month profit measure was positive and large. This distribution of outcomes is typical of many marketing actions. Overall, averaging across customers within a treatment condition, four of the five marketing actions generated positive average profits compared to the ``no action'' control.\footnote{Specifically, for a given action, the uniform policy in which every customer receives that action yields higher average profits than a uniform policy in which every customer receives the `no action'' control.} 

\subsection{Designing Customer Segments} \label{subsec:designing segments}

The volume and similarity constraints in Problem~\eqref{eq:problem_setup} require customer segmentation. We group individual customers into segments in this application using zip codes, which are available for all prospective customers. This offers three benefits.  First, as we discussed in Section \ref{sec:problem}, defining segments using zip codes provides a practical solution to the absence of data about which individual customers fall within a protected class. Second, in the prospecting application that we study, firms do not have past purchasing histories for individual customers. For this reason, segmentation based upon zip codes is particularly common when prospecting for new customers, because census data provides detailed no-cost demographic measures (see for example \citealp{prospective20}).\footnote{In contrast, when targeting existing customers, segmentation is often based upon past purchasing measures, such as the recency, frequency and monetary value of past purchases (\citealp{sahni2017targeted}).  Alternatively, segments could be defined at the store level. A retailer might require that predicted store sales are at least as high as last year, or that the average number of promotions received by customers neighboring each store is similar for each store.} 

Using zip codes to segment customers also helps to address a primary objective of our empirical analysis; observing how well different methods perform when varying the number of constraints. Recall that the number of constraints is a function of the number of customer segments $K$. The hierarchical nature of zip codes provides a convenient way to vary the number of segments. In particular, a four-digit zip code is identified by the first four digits of a five-digit zip code, and contains all of the five-digit zip codes that share those first four digits. As a result, the assignment of five-digit zip codes to four-digit zip codes is mutually exclusive and collectively exhaustive. The same properties apply when we consider three-digit zip codes (or even two-digit and one-digit zip codes). In this application, we define the customer segments using either three-digit zip codes, four-digit zip codes or five-digit zip codes.

For some five-digit zip codes, the number of prospective customers available in that zip code is small. With similarity constraints, trivial optimal solutions are obtained. Thus, in our validation exercise, we focus on five-digit zip codes with at least 4,000 customers in our data. This leaves us $I= 2,065,758$ customers in our validation exercise.\footnote{Recall that our focus is on optimization, rather than minimizing or accounting for estimation errors. Therefore, in our validation exercise, we identify the optimal policy using the same data that we use to predict the response functions (see the discussion in Appendix \ref{subsec:targetingvariables}). In practice, firms may want to minimize over-fitting by designing the optimal policy using a different group of customers than the sample used for estimation.} Our final sample has $K=229$ when we use five-digit zip codes to define customer segments, $K=87$ when we use four-digit zip codes, and $K=18$ when we use three-digit zip codes. When we compare how the different methods perform when we vary the number of customers, we select 25\%, 50\%, and 75\% customers from each customer segment to keep the number of customer segments fixed. 

To solve Problem~\eqref{eq:problem_setup}, we also need to choose exogenous parameters for all of the constraints. We summarize the choice of parameters used in our empirical analyses in Appendix \ref{app:parameters}. With these parameter choices, the number of constraints in our implementation is $2JK+K+K(K-1)J/2 + K(K-1)/2+I$. For example, when we define the customer segments using a five-digit zip code ($K=229$) and consider all customers ($I=2,065,758$), the number of constraints is $2,224,913$.

\subsection{Estimating Response Functions} \label{subsec:estimating response functions}
The randomized assignment of customers in this training data makes it straightforward to use this data to estimate causal treatment effects for each marketing action. In particular, 27 covariates were available to describe each prospective household. This information was purchased by the firm from a third-party commercial data supplier. We use Lasso with a full set of interactions to predict the profit associated with each marketing action for each customer (in this application, a household is equivalent to a customer). In particular, we estimate Lasso six times, representing a separate model for each marketing action (including the no-action control). Although there are many alternative estimators available, Lasso has been shown to be effective in this type of application.\footnote{See for example {\cite{searchtargeting20}} and \cite{prospective20}, who compare the performance of different estimators. Notably, \cite{prospective20} uses a subset of the same data that we use in this study.}

We use these estimates to calculate the incremental predicted profit for each customer and marketing action, by subtracting the predicted profit in the control from the predicted profit associated with that action. In Appendix \ref{subsec:targetingvariables}, we provide additional details, including definitions and summary statistics for the covariates.


\subsection{Benchmark Methods, Hardware and Performance Measures} \label{sec:validation_design}
We consider three benchmark methods: primal simplex, dual simplex, and the barrier method. These are considered state-of-the-art solvers for linear programming problems. We implement each method using Gurobi, which is a commercial software package explicitly designed to solve optimization problems, such as linear programming problems. Data engineers at Gurobi have spent years fine-tuning their implementations of these benchmark methods, and the Gurobi implementations are generally considered amongst the most powerful implementations available.

Different firms have access to different hardware resources. To compare how this affects the performance of the different methods, we compare their performance using three different hardware combinations:  

\begin{itemize}
    \item \textbf{H1}: 8-core CPU and 64GB memory;
    \item \textbf{H2}: 16-core CPU and 128GB memory;
    \item \textbf{H3}: 32-core CPU and 256GB memory.
\end{itemize}

The first hardware option is likely to be feasible and affordable for essentially any firm, because it is representative of the hardware in a standard laptop. The other two specifications include more powerful CPUs and more memory. These hardware configurations recognize that medium and larger firms, and smaller firms with sophisticated technology capabilities, have access to workstations or server-level resources. We implement all hardware options using Google Cloud Computing Services.

We compare the different methods using two performance measures. We first measure feasibility, by asking whether a given method and hardware combination converges to an optimal solution within four days. 
The second measure focuses on computation time, and measures the total time used to solve the problem (from the set of problems that are actually solved). We use these two measures to show that our method can solve larger IPwC problems, and can consistently solve them faster (given the same hardware specification). 

As we point out in Section \ref{sec:problem}, Problem (\ref{eq:problem_setup}) is a convex problem, and Theorem \ref{thm:twoloop} guarantees the convergence of our method. Therefore, any method that can solve the IPwC problem will deliver the same optimal profit. For this reason, we omit the discussion of the optimal profit in this section. However, we turn attention to the optimal profit in Section \ref{sec:manaimplication}, when we illustrate the economic implications of solving larger problems. 

\subsection{Results}\label{sec:results}
Table \ref{table:varyI} reports the results when we vary the number of customers $I$ using hardware $H_3$ and five-digit zip codes to define customer segments. Recall that we randomly select 25\%, 50\%, and 75\% customers from each customer segment to investigate how the number of customers influences the performance of different methods. In this setup, the number of customer segments is fixed at $K=229$. The table reports the total number of seconds required to solve each IPwC problem. If the method cannot solve an instance of the problem, the table reports either ``-'' indicating it ran out of memory, or ``*'' indicating it ran out of time.

\begin{table}[h]
\doublespacing
\centering
\begin{threeparttable}
\caption{Varying the Number of Customers}\vspace{.1in}
\begin{tabular}{ l c c c c} 
 \hline\hline
& Primal Simplex & Dual Simplex  & Barrier & Algorithm \ref{alg:twoloop} \\ 
 \hline
 
  \cline{2-5}
  Full Sample ($I=2,065,758$) & - & - & - & 265,660 \\
 75\% Sample ($I=1,549,323$) &   * & * & * & 180,900 \\
  50\% Sample ($I=1,032,884$)  &  * & * & * & 14,780 \\ 
  25\% Sample ($I=516,435$)  & 44,189 & 10,563 & * & 5,204\\
 \hline\hline
\label{table:varyI}
\end{tabular}
\begin{tablenotes}[para,flushleft]
\item \small \singlespacing \vspace{-.6in} Notes. The table reports the total computation time (in seconds) used by each method to solve each instance of the IPwC problem, when varying the proportion of customers included in each problem. If the method cannot solve an instance of the problem, the table reports either ``-'' indicating it ran out of memory, or ``*'' indicating it ran out of time. 
\end{tablenotes}
\end{threeparttable}
\end{table}


We see that Algorithm \ref{alg:twoloop} solves all of the IPwC problem instances proposed in Table \ref{table:varyI}. In contrast, primal simplex and dual simplex can only solve the problems with the smallest number of customers (25\% of the sample), while the barrier method is unable to solve any of the problems. When primal simplex and dual simplex can solve the problem, they require a lot more computation time than Algorithm \ref{alg:twoloop}. Recall that in Table \ref{table:targetingliterature}, we summarized eight recent marketing papers that studied personalization problems. In four of these papers, the number of customers exceeded 1 million. It is notable that none of the three benchmark methods were able to solve the IPwC problems proposed in Table \ref{table:varyI} when the number of customers exceed 1 million. 

Recall that the findings in Table \ref{table:varyI} were obtained using hardware $H_3$. In Appendix \ref{subsec:numberhouseholds}, we report additional findings when using hardwares $H_1$ and $H_2$. With fewer hardware resources, larger versions of the IPwC problem become unsolvable even with our proposed method. 

With increases in the number of customers, both the number of decision variables and the number of constraints increase. In Table \ref{table:varyK}, we investigate the impact of varying the number of segments ($K$). We use all customers  ($I=2,065,758$) and hardware option $H_3$ in all scenarios. Results using hardware configurations $H_1$ and $H_2$ are reported in Appendix \ref{subsec:numbersegments}.

\begin{table}[h]
\doublespacing
\centering
\begin{threeparttable}
\caption{Varying the Number of Segments}\vspace{.1in}
\begin{tabular}{ l c c c c} 
 \hline\hline
& Primal Simplex & Dual Simplex  & Barrier & Algorithm \ref{alg:twoloop} \\ 
 \hline
 
  \cline{2-5}
  Five-digit zip codes ($K=229$) & - & - & - & 265,660 \\
 Four-digit zip codes ($K=87$) &   * & * & * & 12,210 \\
  Three-digit zip codes ($K=18$)   &  84,679 & 1,304 & 9,375 & 706 \\ 
 \hline\hline
\label{table:varyK}
\end{tabular}
\footnotesize
\begin{tablenotes}
\item \singlespacing \vspace{-.45in} Notes. The table reports the total computation time (in seconds) used by each method to solve each instance of the IPwC problem, when varying the coarseness of the customer segmentation. If the method cannot solve an instance of the problem, the table reports either ``-'' indicating it ran out of memory, or ``*'' indicating it ran out of time. 
\end{tablenotes}
\end{threeparttable}
\end{table}


In this comparison, Algorithm \ref{alg:twoloop} again solves all of the problems, while the benchmark methods only obtain solutions when the segmentation is relatively coarse (using 3-digit zip codes, for which $K=18$). In this scenario, we again see that our proposed method has a much faster computation time. More generally, when the number of segments decreases, the computation time for Algorithm \ref{alg:twoloop} decreases very quickly. 

In our next set of comparisons, we investigate what kind of constraints are most challenging for the different methods. In particular, we investigate how well the different methods perform when the IPwC problem contains: (a) only Volume I and II constraints, (b) only Similarity I and II constraints, or (c) the combination of all of these constraints. All scenarios use hardware option $H_3$, all of the customers ($I=2,065,758$), and five-digit zip codes to define customer segments ($K=229$). 

\begin{table}[h]
\doublespacing
\centering
\begin{threeparttable}
\caption{Varying the Type(s) of Constraints}\vspace{.1in}
\begin{tabular}{ l c c c c} 
 \hline\hline
& Primal Simplex & Dual Simplex  & Barrier & Algorithm \ref{alg:twoloop} \\ 
 \hline
 
  \cline{2-5}
 Both Types of Constraints    &  - & - & - & 265,660 \\ 
 Only Similarity Constraints &   - & - & - & 226,700 \\
 Only Volume Constraints & 3,233 & 205 & 270 & 228\\
\hline

 \hline\hline
\label{table:constraint}
\end{tabular}
\footnotesize
\begin{tablenotes}
\item \singlespacing \vspace{-.45in} Notes. The table reports the total computation time (in seconds) used by each method to solve each instance of the IPwC problem, when varying the type(s) of constraints included in the problem. If the method cannot solve an instance of the problem, the table reports either ``-'' indicating it ran out of memory, or ``*'' indicating it ran out of time.  
\end{tablenotes}
\end{threeparttable}
\end{table}

The findings in Table \ref{table:constraint} reveal that similarity constraints appear to introduce a more formidable challenge to IPwC problems than volume constraints. If a firm only wants to impose volume constraints, then even with many customers and many segments, the benchmark methods can solve the problem in a reasonable amount of time (3,233 seconds is just less than one hour). However, if the firm wants to incorporate similarity constraints, perhaps due to fairness concerns, then the problem becomes too difficult for these methods. It is these settings in which our proposed algorithm will be particularly useful. 

Together, the empirical application described in this section reveals that our proposed method can extend the scale of IPwC problems that can be solved. This includes expanding the number of customers and customer segments that can be considered, and enabling the inclusion of similarity constraints, rather than just volume constraints. In the next section, we measure the implications for firms, by comparing how adjusting the problem to make it solvable with the benchmark methods affects the profitability of the resulting policy.

\section{Economic Implications}\label{sec:manaimplication}
In this section, we illustrate the economic importance of our proposed method by measuring the increase in expected profit that a firm can earn by using this method to solve the IPwC problem, compared to a more restricted problem that can be solved using the benchmark methods. We first introduce an alternative version of the targeting problem with constraints, in which actions are assigned at the segment level. The benchmark methods can all solve this more restricted problem. We compare the computation cost and optimal profit differences between the IPwC problem and this problem. We then adjust the level of the segmentation in the alternative problem, so that it requires the same computation time as the IPwC problem. This allows us to compare the profit difference when using an equivalent computation time.

\subsection{A More Restricted Version of the Targeting Problem with Constraints}

In Section \ref{sec:validation}, we showed that Algorithm \ref{alg:twoloop} can solve the IPwC problem in our empirical application using the complete set of customers, a large number of segments (constructed at the 5-digit zip code level), and the complete set of volume and similarity constraints. In contrast, this problem could not be solved by any of the three benchmark methods.

If a firm did not have access to Algorithm \ref{alg:twoloop}, and instead was forced to use one of the benchmark methods, it would have to adjust the problem. One solution would be to reduce the number of number of customers, number of segments, or number of constraints. However, this fundamentally changes the firm's problem. If the firm wants to impose similarity constraints on the full set of customers and segments, then taking any of these steps to make the problem solvable would mean that the solution is not guaranteed to be an optimal or even feasible solution to the problem the firm actually wants to solve.

An alternative way to make the problem solvable is to reduce the number of decision variables, by requiring that some customers receive the same marketing actions. In particular, the firm could choose a different marketing action for each customer segment, but require that all customers within a segment receive the same action. We can formulate this segment-level policy as follows: 

\begin{align}
    \max_{x_{k^*}^j} & \ \ \ \sum_{k^*=1}^{K^*} \sum_{j=1}^J p_{k^*}^j x_{k^*}^j \nonumber \\
    \text{s.t.} &\ \ \ a_k^j \le \sum_{i\in S_{k}}  x_{k^*}^j \le b_k^j, \  \text{ for } j=1,...,J, k=1,...,K \ \ \ \textbf{(Volume I)}   \nonumber \\
    & \ \ \  L_k \le \sum_{i\in S_{k}}  \sum_{j=1}^J c^j_{k^*} x_{k^*}^j \le U_k, \ \text{ for } k=1,...,K \ \ \ \textbf{(Volume II)}  \nonumber \\
    & \ \ \  \frac{1}{n_{k_1}} \sum_{i\in S_{k_1}} x_{k^*}^j \le \lambda_j^{k_{1}k_{2}} \frac{1}{n_{k_2}} \sum_{i\in S_{k_2}} x_{k^*}^j + g_j^{k_{1}k_{2}}, \nonumber \\
    &\ \ \  \text{ for } j=1,...,J,  k_1=1,...,K, k_2=1,...,K, k_1\neq k_2 \ \ \ \textbf{(Similarity I)} \nonumber \\
     & \ \ \  \frac{1}{n_{k_1}} \sum_{i\in S_{k_1}}\sum_{j=1}^J d_{k^*}^j x_{k^*}^j \le \gamma^{k_{1}k_{2}} \frac{1}{n_{k_2}} \sum_{i\in S_{k_2}}\sum_{j=1}^J d_{k^*}^j x_{k^*}^j + h^{k_{1}k_{2}},  \nonumber \\
    &\ \ \  \text{ for } k_1=1,...,K, k_2=1,...,K, k_1\neq k_2  \ \ \ \textbf{(Similarity II)} \nonumber \\
    & \ \ \  \sum_{j=1}^J x_{k^*}^j \le M_{k^*},  \ \text{ for } k^*=1,...,K^*  \ \ \ \textbf{(Targeting)}\nonumber \\  
    & \ \ \ 0 \le x_{k^*}^j \le 1 \ \ \ \textbf{(Feasibility)} \ .
    \label{eq:problem_setup_segment}
\end{align}
We will describe this as the ``Segment Personalization with Constraints'' problem, or ``SPwC''. Notice that in this problem, customers are segmented in two different ways. We use $k$ to denote the segmentation used in the volume and similarity constraints, and $k^*$ to denote the segmentation used for grouping customers when assigning marketing actions. Specifically, a customer is assigned to both a $k$ segment and separately to a $k^*$ segment. These two segmentations could be identical, but they could also vary. For example, a firm might want to use zip codes to construct segments for the similarity constraints (see earlier discussion), but when assigning marketing actions, segment customers so they align with the firm's production systems (e.g. sales person territories). 

Compared to the IPwC, the decision variables in this problem change to $x_{k^*}^j\in [0,1]$ to represent the probability that customers in segment $k^*$ receive marketing action $j$. The formulation allows boundary solutions, in which all customers in segment $k^*$ receive treatment $j$ ($x_{k^*}^j = 1)$, or none of the customers in segment $k^*$ receive treatment $j$ ($x_{k^*}^j = 0)$. More generally, $x_{k^*}^j$ represents the proportion of customers in segment $k^*$ who will receive marketing action $j$. 

The $p_{k^*}^j$ term denotes the incremental profit that the firm earns from customer segment $k^*$ if it receives marketing action $j$. This is calculated as the sum of $p_i^j$ for all customers $i$ in segment $k^*$. The constraints in \eqref{eq:problem_setup_segment} are otherwise the same as the constraints in \eqref{eq:problem_setup}. The only difference between the two problems is that \eqref{eq:problem_setup} designs an individual-level targeting policy ($x_i^j$), while \eqref{eq:problem_setup_segment} designs a segment-level targeting policy ($x_{k^*}^j$). 


We implemented three separate versions of both problems (IPwC and SPwC), using the complete dataset ($I = 2,065,758)$, and the complete set of volume and similarity constraints. The three versions vary in the $k$-level segmentation used to define the volume and similarity constraints. Consistent with the analysis in Table \ref{table:varyK}, we defined these segments using zip codes identified at the 5-digit level ($K = 229$), 4-digit level ($K = 87$), or 3-digit level ($K = 18$).  For all three problems, the $k^*$-level segmentation used to segment the marketing actions in the SPwC problem was defined at the 5-digit zip code level. 

As we previously discussed, the problem is convex, and so any method that can solve a specific version of a problem, obtains the same optimal profit. However, the optimal profit varies within the pairs of IPwC and SPwC problems, and we summarize these profit differences in Table \ref{table:profit}. In the first row, we report the profit difference as an average per 100 customers. In the second row we calculate the difference in the computation cost when solving an IPwC problem using Algorithm \ref{alg:twoloop} versus solving the associated SPwC problem using the barrier or simplex methods. The computation costs were calculated using Google Cloud computing expenses, and are reported per 100 customers.\footnote{\label{footnote:comp time}When solving the IPwC problem using Algorithm \ref{alg:twoloop}, we divide the total computation cost by the total number of customers, and then multiple by 100. All of the SPwC problem can be solved using the barrier or simplex methods in one or two seconds, and so we treat this computation cost as zero.}

\begin{table}[h]
\doublespacing
\centering
\begin{threeparttable}
\caption{Optimal Profit: IPwC Compared to SPwC}\vspace{.2in}
\begin{tabular}{ l c c  c} 
 \hline\hline
   & 3-digit $k$-segments & 4-digit $k$-segments & 5-digit $k$-segments\\
   & $K = 18$  & $K = 87$ & $K = 229$\\
  \hline
  Optimal Profit Difference & \$$16.064$ & \$$19.124$ & \$$23.302$ \\
   Computation Cost Difference  & -\$$0.001$ & -\$$0.002$ &  -\$$0.012$ \\
   Total Benefit & \$$16.063$ & \$$19.122$ & \$$23.290$  \\
 \hline\hline
\end{tabular}
\begin{tablenotes}[para,flushleft]
\item \small \singlespacing \vspace{-.3in} Notes. The first row reports the difference in the optimal profit between the IPwC and SPwC problems (using the same constraints for each pair of problems). Profits are calculated as the average profit per 100 customers. The second row reports the difference in computation cost from solving the IPwC problem using Algorithm \ref{alg:twoloop} and the SPwC problem using the simplex or barrier methods. The computation costs are measured in terms of Google Cloud computing expenses, and are indexed to a cost per 100 customers. 
\end{tablenotes}
\label{table:profit}
\end{threeparttable}
\end{table}

There are several findings of interest. First, the optimal profit difference is positive, indicating that the optimal profit is higher for the IPwC problem than for the corresponding SPwC problem. This is what we would expect, because the two problems have identical sets of constraints, but the IPwC problem has more degrees of freedom. Any solution to the SPwC problem is also a feasible solution to the IPwC problem. As a result, the optimal solution to the IPwC problem is guaranteed to be (weakly) larger than the optimal solution to the SPwC problem.

Second, the optimal profit difference grows with the complexity of the constraints. When customers are segmented using 3-digit zip codes for the volume and similarity constraints, the optimal profit difference (between the IPwC and SPwC problems) is \$16.064, which grows to \$23.302 when segmenting at the 5-digit zip code level. The additional degrees of freedom in the IPwC problem provide it with more opportunities to find a solution that satisfies the additional complexity in the constraints. We caution that while this finding is perhaps intuitive, we have not established that it will generalize to all applications.

Third, we see that the computation cost difference is negative for all three pairs of problems. This indicates that the computation cost is higher when solving the IPwC problem using Algorithm \ref{alg:twoloop}, than when solving the SPwC problem using the simplex or barrier methods. Notably, the computation cost differences are trivial compared to the optimal profit differences. This suggests that if a firm was restricted to using the simplex or barrier methods, it may be profitable to invest in additional computation time, in order to solve versions of the SPwC problem with less coarse segmentation of the marketing actions. The tradeoff between optimal profit and computation time appears to strongly favor investing in additional computation time. 


This last result introduces a new question: if the (IPwC - Algorithm \ref{alg:twoloop}) and (SPwC - barrier or simplex) problems used an equivalent amount of computation, what would the optimal profit difference be? We address this question next.

\subsection{Difference in Optimal Profits Using the Same Computation Time}
In this analysis, we use 3-digit zip codes to define the $k$-level segmentation for the volume and similarity constraints, so that these constraints are identical in both the IPwC and SPwC problems. Recall from Table \ref{table:varyK}, that the computation time required by Algorithm \ref{alg:twoloop} to solve this IPwC problem is 706 seconds. In the SPwC problem, we vary the $k^*$-level segmentation used to group customers when assigning actions, so that solving the SPwC problem also requires 706 seconds (using the dual simplex method).

More specifically, when each customer is in its own $k^*$ segment, so that actions are chosen separately for each customer, then the SPwC and IPwC problems are equivalent. We know that with this level of $k^*$ segmentation, the dual simplex method requires 1,034 seconds to solve the SPwC problem (Table \ref{table:varyK}). In contrast, if the $k^*$ segmentation is at the 5-digit zip code level, then the dual simplex method requires just seconds to solve the problem (Footnote \ref{footnote:comp time}). Therefore, we seek an intermediate level of $k^*$ segmentation, between individual customers and 5-digit zip codes, in which the dual simplex method requires approximately 706 seconds to solve the SPwC problem. We can then compare the optimal profit from the (IPwC - Algorithm \ref{alg:twoloop}) and (SPwC - dual simplex) solutions, where the two solutions are each obtained using the same amount of computation.


To vary the coarseness of the $k^*$ segmentation in the SPwC problem, we randomly select some 5-digit zip codes in which we segment at the zip code level, and in the remaining 5-digit zip codes, we segment at the individual customer level. Where we segment at the zip code level, all households within that zip code are assigned the same action. Where we segment at the customer level, actions are assigned separately to individual customers. The larger the proportion of 5-digit zip codes that we assign actions at the zip code level, the more coarse the $k^*$ segmentation, and the fewer the degrees of freedom available to the SPwC problem. The optimal profit from the SPwC problem will (weakly) decrease, and the required computation time will also decrease. 

We illustrate the findings in Figure~\ref{fig:carrier}, where each data point represents an average across 30 iterations of the SPwC problem (with different random draws of zip codes in each iteration). The X-axis varies the proportion of randomly selected zip codes in which the SPwC problem assigns actions at the zip code level (in the remaining segments actions are assigned at the customer level). The black line reports the average computation time for the SPwC problem (using dual simplex). The columns report the average profit difference per 100 customers between the (IPwC - Algorithm \ref{alg:twoloop}) and (SPwC - dual simplex) solutions.

As expected, the IPwC problem is always more profitable than the SPwC problem (the profit difference is positive). Moreover, increasing the proportion of segments in which all customers receive the same marketing action, both decreases the computation time for the SPwC problem (black line), and increases the profit advantage for the IPwC problem over the SPwC problem (columns). Recall that the computation time required to solve the IPwC problem using Algorithm \ref{alg:twoloop} is 706 seconds. When the SPwC assigns actions at the zip code level in 21\% of the zip codes, dual simplex requires a nearly identical amount of time (689 seconds) to solve the SPwC problem. At this level of segmentation, the IPwC problem earns \$3.01 more profit per 100 customers.

\begin{figure}[h]
\begin{center}
\caption{IPwC Compared to SPwC}
\includegraphics[width=0.7\textwidth]{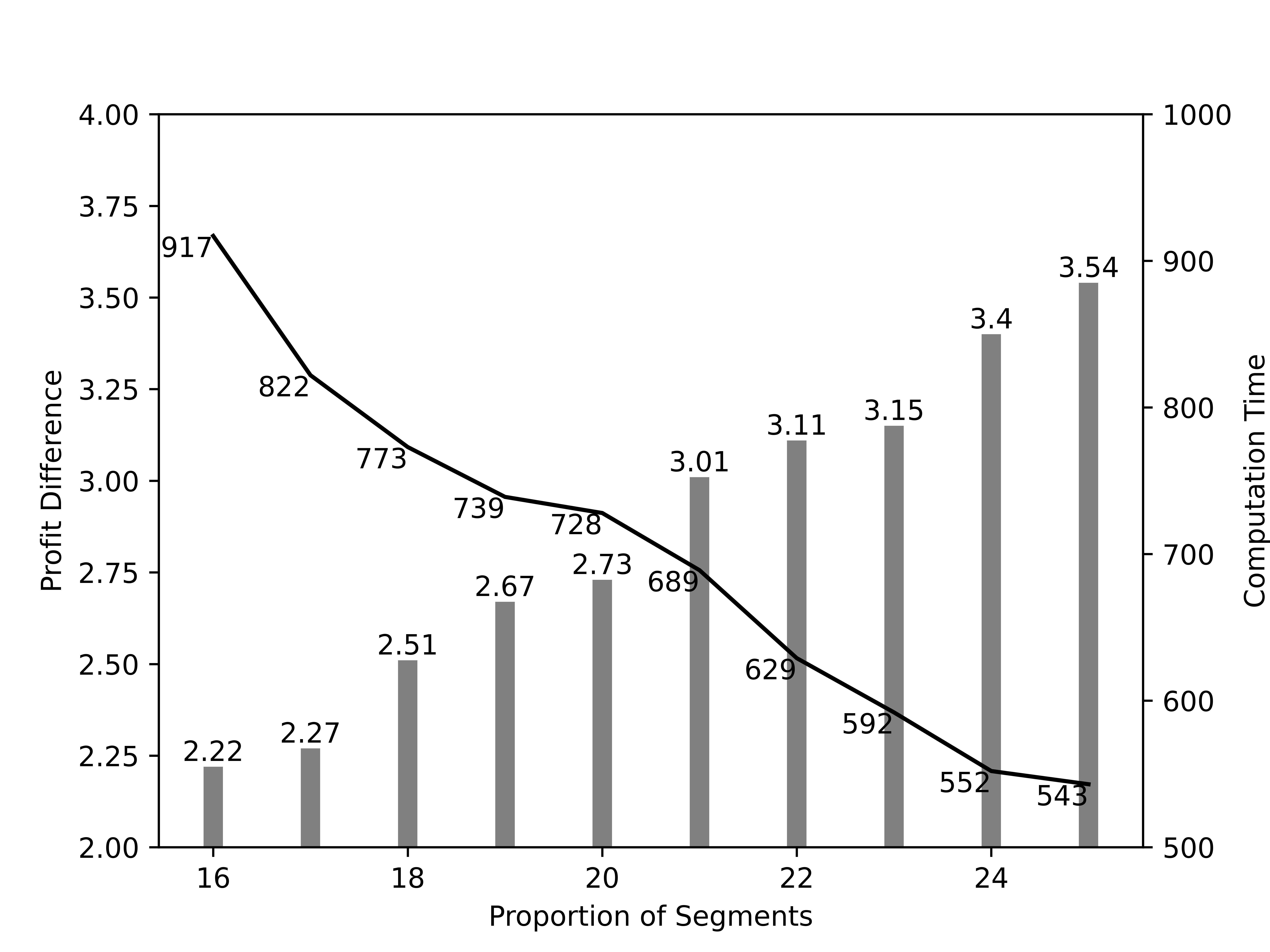}
\label{fig:carrier}
\end{center}
\centerline{\begin{minipage}{.75\textwidth}
\vspace{0.1cm} 
\small  Notes: Each data point represents an average across 30 iterations of the SPwC problem (with different random draws of zip codes in each iteration). The X-axis varies the proportion of randomly selected zip codes in which the SPwC problem assigns actions at the zip code level. The black line reports the average computation time for the SPwC problem (using dual simplex). The columns report the average profit difference per 100 customers between the (IPwC - Algorithm \ref{alg:twoloop}) and (SPwC - dual simplex) solutions.
\end{minipage}}
\end{figure}


The findings in Table \ref{table:profit} illustrate that the additional degrees of freedom that Algorithm \ref{alg:twoloop} can exploit, may result in substantial profit increases. We caution that the optimal profit differences between the SPwC and IPwC problems are specific to the constraints and parameters that we used. However, we expect that the key implications from Table \ref{table:profit} will generalize: (a) the SPwC problem yields lower optimal profits than the IPwC problem, and (b) Algorithm \ref{alg:twoloop} can solve larger versions of the IPwC problem than the benchmark methods. We conclude that the algorithm has the potential to contribute to economically important increases in the profitability of firms' personalization policies.

\section{Conclusion}\label{sec:conclusion}
Much of the recent research in marketing using machine learning has focused on new methods for estimating customer response functions. Our paper takes a step in a different direction: using recent advances in optimization methods to help firms optimize policies once they have estimated those response functions. We propose a method for optimizing large-scale personalization problems in the presence of constraints. 

We focus on two types of constraints. Volume constraints restrict the total number of marketing actions that can be taken, either through (predetermined) minimum or maximum thresholds. Similarity constraints limit the difference in the frequency of marketing actions taken with different customer segments, and are often motivated by concerns for fairness.  

The proposed method departs from existing state-of-the-art methods by using first-order methods for linear programming to increase scalability. The algorithm overcomes the challenge that first-order methods quickly find moderately accurate solutions, but then slow down. To address this limitation, the proposed method uses a two-loop primal-dual hybrid gradient (PDHG) algorithm. 

We provide theoretical guarantees on the performance of the proposed method for personalization problems with constraints. First, we show that our proposed method requires fewer computations per iteration than state-of-the-art benchmark methods (primal simplex, dual simplex and barrier methods). Second, we adapt existing guarantees on optimality and computation speed, by adjusting the proofs to accommodate the features of personalization problems. 

In an empirical application, we compare the proposed method to the three benchmark methods. Our comparison focuses on both the size of the problems that can be solved by each method, and the computation time required to reach the optimal solution. The proposed method greatly expands the size of the personalization problems that can be solved, particularly when personalization problems include similarity constraints. Incorporating similarity constraints is especially challenging for the benchmark methods.

The expansion in the size of the problems that are now feasible includes increases in the number of customers, number of customer segments, and number of constraints. Across all of these problems, the proposed method required much less computation time to find the optimal solution compared to any of the benchmark methods. Together, our theoretical and empirical results confirm that designing large-scale personalization policies with constraints is now feasible. 

Many interesting problems remain, and these offer promising avenues for future research. First, as we mentioned in Section \ref{sec:problem}, our approach remains within the predict-then-optimize paradigm. This framework has a potential limitation: the estimation goal is not always the same as the optimization goal. \cite{chunprofits20} propose one approach to address this misalignment when the personalization problem has no constraints. Future research could investigate how to address this misalignment in the presence of constraints. Second, because we use finite sample datasets to estimate customer response functions, these response estimates are estimated with error. The errors will affect the performance of  optimization methods that rely upon those estimates. Future research could investigate how to mitigate the cost of these errors in the optimization step.

\newpage
\singlespacing
\bibliographystyle{chicago} 
\bibliography{targeting}

\newpage 
\setcounter{table}{0} \setcounter{figure}{0} %
\setcounter{equation}{0} \setcounter{subsection}{0}
\setcounter{page}{1} %
\setcounter{footnote}{0} %

\renewcommand{\thetable}{A\arabic{table}} %
\renewcommand{\thefigure}{A\arabic{figure}} %
\renewcommand{\theequation}{A\arabic{equation}}%
\renewcommand{\thefootnote}{A\arabic{footnote}} %
\renewcommand{\thepage}{Appendix Page \arabic{page}}
\renewcommand{\thesubsection}{A.\arabic{subsection}}%
\renewcommand{\thesubsection}{A.\arabic{subsection}}

\doublespacing
\section*{Appendix}
\subsection{Alternative Problem Setup}\label{subsec:setup_another}
In this section, we show how a personalization problem with constraints can be modeled as a linear programming problem if we assume there exists interdependence across different marketing actions. We use a case where interdependence can exist for at most two marketing actions to illustrate. 

Let $y^{j_1, j_2}_i$ represent the use of marketing actions $j_1, j_2$ together on customer $i$, $z_i^j$ represent just use marketing action $j$ on customer $i$. $p_i^j$ is the incremental profit if customer $i$ only receives marketing action $j$, and $q_i^{j_1, j_2}$ as the incremental profit if customer $i$ receives both marketing action $j_1$ and $j_2$. $x_i^j$ is an auxiliary variable. We can model the problem as the LP model in \ref{eq:problem_setup2}.

To understand Problem \ref{eq:problem_setup2}, the key parts are targeting and auxiliary variable conditions. What the targeting condition illustrates is that if the customer $i$ only receives marketing action $j$, $z_i^j=1$ and $y_i^{j_1,j_2} = 0$. If customer $i$ receives marketing actions $j_1$ and $j_2$, $z_i^j=0$ and $y_i^{j_1,j_2} = 1$. Here, to simplify the problem expression, we restrict that each customer can at most receive two marketing actions. Next, based on the auxiliary variable condition, $x_i^j$ represents the probability a given customer $i$ receives marketing action $j$. This probability is a sum of the probability customer $i$ only receives marketing action $j$ and the probability customer $i$ receives two marketing actions with one of them as marketing action $j$. With this setup, all other constraints (volume and similarity constraints) have the same meaning as the one that we describe in Section \ref{sec:problem}.

The problem is still a linear programming, and our proposed method can be directly applied.

\begin{align}
    \max_{x,y,z} & \ \ \ \sum_{i=1}^I \left(\sum_{j=1}^J p_i^j z_i^j + \sum_{1\le j_1<j_2\le J} a^{j_1, j_2}_i y^{j_1, j_2}_i\right)  \nonumber \\
    \text{s.t.} 
    &\ \ \ a_k^j \le \sum_{i\in S_k} x_i^j \le b_k^j, \  \text{ for } j=1,...,J, k=1,...,K \ \ \ \textbf{(Volume I)}   \nonumber \\
    & \ \ \  L_k \le \sum_{i\in S_k} \sum_{j=1}^J c^j_i x_i^j \le U_k, \ \text{ for } k=1,...,K \ \ \ \textbf{(Volume II)}  \nonumber \\
    & \ \ \  \frac{1}{n_{k_1}} \sum_{i\in S_{k_1}} x_i^j \le \lambda_j^{k_{1}k_{2}} \frac{1}{n_{k_2}} \sum_{i\in S_{k_2}} x_i^j + g_j^{k_{1}k_{2}}, \nonumber \\
   & \ \ \ \text{ for } j=1,...,J,  k_1=1,...,K, k_2=1,...,K, k_1 \neq k_2 \ \ \ \textbf{(Similarity I)} \nonumber \\
     & \ \ \  \frac{1}{n_{k_1}} \sum_{i\in S_{k_1}}\sum_{j=1}^J d_i^j x_i^j \le \gamma^{k_{1}k_{2}} \frac{1}{n_{k_2}} \sum_{i\in S_{k_2}}\sum_{j=1}^J d_i^j x_i^j + h^{k_{1}k_{2}}, \nonumber \\
     & \ \ \ \text{ for } k_1=1,...,K, k_2=1,...,K, k_1 \neq k_2  \ \ \ \textbf{(Similarity II)} \nonumber \\
    & \ \ \  \sum_{j=1}^J z_i^j+ \sum_{1\le j_1<j_2\le J} y_i^{j_1,j_2}\le 1,  \ \text{ for } i=1,...,I  \ \ \ \textbf{(Targeting)}\nonumber \\
    & \ \ \  x_i^j\ge y_i^{j, j_2},  x_i^j\ge y_i^{j_1, j}, x_i^j\ge z_i^{j} \ \text{ for } i=1,...,I, j=1,...,J, 1\le j_1<j, j< j_2\le J   \nonumber \\
    & \ \ \  x_i^j= z_i^{j}+\sum_{j_2>j} y_i^{j, j_2} +  \sum_{j_1<j} y_i^{j_1,j} \ \text{ for } i=1,...,I, j=1,...,J    \ \ \ \textbf{(Auxiliary Varible)} \nonumber \\
    & \ \ \ 0\le x_i^j \le 1, 0\le y_i^{j_1,j_2}\le 1, 0\le z_i^j\le 1 \ \ \ \textbf{(Feasibility)} \ .
    \label{eq:problem_setup2}
\end{align}

\subsection{Details of Algorithm \ref{alg:twoloop}}\label{subsec:algorithm_detail}

For notational convenience, we denote $z=(x,y)$ as the combined primal-dual variable, and $Z=X\times Y$ as the constraint set for $z$. We adopt the normalized duality gap in \cite{applegate21b}:
\begin{definition}
Denote $(\tx,\ty) = \tz$ and the objective in Equation \eqref{eq:problem_primal_dual} as $L(x,y)$. We define the normalized duality gap at solution $z$ with radius $r$ as 
\begin{align}\label{eq:rhorz-pos}
\rho_{r}(z) := \frac{\max_{\tz \in B_{r}(z)\cap Z}\{ L(x, \ty) - L(\tx, y)\} }{r} \ ,
\end{align}
where $B_{r}(z) :=\{\tz \in Z | \|z-\tz\|\le r\}$ is the ball centered at $z$ with radius $r \in (0, \infty)$ intersected with the set $Z$, and $\| \cdot \|$ is a norm on $Z$.
\end{definition}

The normalized duality gap decay condition referred in Algorithm \ref{alg:twoloop} is defined as whenever the normalized duality gap halves:
\begin{equation*}
     \rho_{||\bar{z}^{n,t}-z^{n,0}||}(\bar{z}^{n,t}) \le 0.5  \rho_{||z^{n,0}-z^{n-1,0}||}(z^{n,0}).
\end{equation*}


\subsection{Proof of Theorem \ref{thm:twoloop}}\label{subsec:largescale_proof}

We here consider the primal-dual form of linear program \eqref{eq:problem_primal_dual}. The proof of Theorem \ref{thm:twoloop} is based on the results in recent paper \cite{applegate21b}. The major difference between our setting and \cite{applegate21b} is that we have a box constraint $x\in X=\{0\le x \le e\}$ on the primal variable, while the problem studied in \cite{applegate21b} considers unbounded constraint $x\in X=\{x\ge 0\}$. As a result, the sharpness results in Section 3.3 in \cite{applegate21b} do not readily apply here. Indeed, it is highly challenging to extend their analysis based on Hoffman constant into our setting, because the dual problem of \eqref{eq:problem_primal_dual} is an unconstrained minimization problem with a piecewise-linear objective function, and the KKT system studied in \cite{applegate21b} is no longer valid. Instead, we here present a very different analysis to show that \eqref{eq:problem_primal_dual} is also a sharp problem. Then we are immediately able to utilize Theorem 1 in \cite{applegate21b} to prove Theorem \ref{thm:twoloop}.

The next theorem shows that \eqref{eq:problem_primal_dual} is a sharp problem on any bounded region $S$.
\begin{theorem}\label{prop:sharp}
Suppose the primal problem \eqref{eq:problem_primal} has a finite optimal solution. Then \eqref{eq:problem_primal_dual} is a sharp primal-dual problem, namely, for any $r>0$ and $R>0$, there exists $\alpha>0$ such that it holds for any $z\in B_R(0)$ that
$$
\rho_r(z)\ge \alpha \dist(z, Z^*) \ ,
$$
where $Z^*$ is the optimal solution set to \eqref{eq:problem_primal_dual}, and $\dist(z, Z^*)=\min_{z^*\in Z^*} \|z-z^*\|_{\infty}$ is the distance between $z$ and the optimal solution set $Z^*$ in the $\ell_{\infty}$ norm.
\end{theorem}



Without loss of generality, we can assume the norm used in $B_r(z)$ is the infinity norm in $x$ and in $y$, i.e., $B_r(z)=\{(\tx,\ty): \|x-\tx\|_{\infty} \le r, \|y-\ty\|_{\infty}\le r\}$, because all norms are equivalent up to a constant in a finite Euclidean space. In other word, the existence of such $\alpha>0$ in Theorem \ref{prop:sharp} keeps valid with a different choice of norm, although the constant $\alpha$ may vary with a different choice of the norm. Similarly, we define the norm in the primal space and in the dual space as $\ell_{\infty}$ norm.

To prove the theorem, we utilize the following two lemmas.

\begin{lemma}\label{lem:gamma}
Denote $\mF=\{0\le x \le e:Gx\le h\}$ as the feasible region of \eqref{eq:problem_primal}. Suppose $\mF$ is non-empty, then there exists $\gamma>0$ such that it holds for any $0\le x\le e$ that
$$
\|(Gx-h)^+\|_1\ge \gamma \text{dist}(x,\mF) \ ,
$$
where the distance (\textit{dist}) is defined with $\ell_{\infty}$ norm and $x\in [0,1]^n$.
\end{lemma}

\begin{proof}
Define $f(x)=\|(Gx-h)^+\|_1+1_{0\le x\le e}$ where $1_{0\le x\le e}=\left\{\begin{array}{cc}
    0 & \text{ if } 0\le x\le e \\
    \infty & \text{ otherwise}
\end{array}\right.$ is the indicator function of the set $[0,1]^n$. Since $\mF$ is non-empty, we know that $f^*=\min_{x} f(x)=0$ and the optimal solution set to this function $f(x)$ is $\mF$. 
Furthermore, notice that $f(x)$ is a piecewise linear function, thus it is a sharp function, namely there exists $\gamma>0$ such that for any $0\le x\le e$
\begin{equation}
\|(Gx-h)^+\|_1=f(x) = f(x)-f^*\ge \gamma \textbf{\textit{\text{dist}}}(x,\mF) \ ,
\end{equation}
which finishes the proof.

\end{proof}

\begin{lemma}\label{lem:beta}
Suppose $r > \frac{1}{\gamma}\|p\|_1$, then there exists $\beta>0$ such that it holds for any $0\le x\le e$ that
$$
p^T x +r\|(Gx-h)^+\|_1-P^* \ge \beta \dist(x, X^*) \ ,
$$
where $P^*$ is the optimal value to the primal problem \eqref{eq:problem_primal} and $X^*$ is the optimal solution set to \eqref{eq:problem_primal}.
\end{lemma}

\begin{proof}
Consider a function $f_r(x)=p^T x+r\|(Gx-h)^+\|_1+1_{x\in [0,1]^n}$. 
For $x^1\in\arg\min f_r(x)$, suppose $x^1\not\in \mF$. Denote $x^2=\arg\min_{x\in \mF} \|x^1-x^2\|_{\infty}$. Then it follows
from Lemma \ref{lem:gamma} that
$$\|(Gx-h^1)^+\|_1 \ge \gamma \|x^1-x^2\|_{\infty} \ . $$
Notice that $r > \frac{1}{\gamma} \|p\|_1$, thus we have
$$
p^T x^1 + r\|(Gx-h^1)^+\|_1 >  p^T x^1 + \|p\|_1 \|x^1-x^2\|_{\infty}\ge p^T x^2\ .
$$
Thus $f_r(x^2)=p^T x^2 < f_r(x^1)$, which contradicts with the fact that $x^1\in\arg\min f_r(x)$. This shows that any minimizer $x^1$ to $f_r(x)$ must satisfy $Gx^1\le h$, thus $x^1$ is a feasible solution to \eqref{eq:problem_primal}, whereby it is a minimizer to the primal problem \eqref{eq:problem_primal} and $f_r^*=P^*$. Let $X^*_r$ be the set of minimizers of $f_r(x)$, then the above argument implies that $X^*_r\subseteq X^*$.

Notice that $f_r(x)$ is a piecewise linear function thus it is a sharp function, namely, there exists $\beta>0$ such that
$$
f_r(x)-f_r^* \ge \beta \dist(x, X_r^*) \ .
$$
Substituting $f_r(x)=p^T+r\|(Gx-h^1)^+\|_1+1_{x\in [0,1]^n}$ and $f_r^*=P^*$, we have for any $x\in [0,1]^n$ that
$$
p^T x+r\|(Gx-h^1)^+\|_1-P^* \ge \beta \dist(x, X_r^*)\ge \beta\dist(x, X^*)\  ,
$$
which finishes the proof.


\end{proof}

\begin{proof}[Proof of Theorem \ref{prop:sharp}]

First, we consider the case when $r\ge \{\gamma\|p\|, R, 1\}$, then we have
\begin{align}\label{eq:middle}
    r\rho_r(z) &= \max_{\ty\in B_r(y)\cap Y} L(x,\ty) - \min_{\tx\in B_r(x)\cap X} L(\tx,y) \\
    &= \max_{\ty\in B_r(y)\cap Y} \ty^T (Gx-h) + p^T x - \min_{\tx\in B_r(x)\cap X} \pran{\tx^T (p+G^T y) - h^T y} \\
    &\ge r  \|(Gx-h)^+\|_1 + p^T x - \pran{\|(p+G^T y)^+\|_1- h^T y}\\
    & = \underbrace{ r  \|(Gx-h)^+\|_1 + p^T x - p^T x^*}_{I} + \underbrace{\|(p+G^T y^*)^+\|_1- h^T y^* - \|(p+G^T y)^+\|_1+ h^T y}_{II}\ .
\end{align}
The first part in the first equality uses $r\ge R \ge \|y\|_{\infty}$, thus an optimal solution to the maximization problem is $\ty^*_i=\left\{ \begin{array}{cc}
     0 & G_i x-h_i\le 0 \\
     y_i+r & G_i x-h_i > 0
\end{array}\right. \ ,$ whereby the optimal objective value is larger than $r \|(Gx-h)^+\|_1 + p^T x$. The second part in the first inequality uses $r\ge 1$, thus an optimal solution to the minimization problem is $\tx^*_i=\left\{ \begin{array}{cc}
     0 & p_i + G_{i:}^T y\le 0 \\
     1 & p_i + G_{i:}^T y >0
\end{array}\right. \ ,$ whereby the optimal objective value is $ \|(p + G^T y)^+\|_1-h^T y$. The last equality uses strong duality on the primal-dual pair \eqref{eq:problem_primal} and \eqref{eq:problem_dual}, thus $p^T x^*=\|(p+G^T y^*)^+\|_1- h^T y^*$.

It then follows from Lemma \ref{lem:beta} that 
$$
I \ge \beta \dist(x, X^*) \ .
$$
Furthermore, notice that the dual function $D(y):=\|(p+G^T y)^+\|_1- h^T y$ is a piecewise linear function, and $Y^*$ is the optimal solution set to $\max_{y\ge 0} D(y)$, thus there exists $\theta>0$ such that 
$$
II \ge \theta \dist(y, Y^*) \ .
$$
Substituting the above two inequalities to \eqref{eq:middle}, we show that 
$$
\rho_r(z) \ge \frac{1}{r}\min(\beta, \theta) \dist(z, Z^*).
$$

The above shows that $\rho_r(z)$ is $\frac{1}{r}\min(\beta,\theta)$ sharp for $r\ge \{\gamma\|p\|_1, R, 1\}$. For the case when 
$r< \{\gamma\|p\|_1, R, 1\}$, notice that $\rho_r(z)$ is a monotonic non-increasing function in $r$ (Fact 2 in \citealt{applegate21b}), we know that $\rho_r(z)$ is sharp with $\alpha=\frac{1}{\{\gamma\|p\|_1, R, 1\}}\min(\beta,\theta)$. This finishes the proof.
\end{proof}

\begin{proof}[Proof of Theorem \ref{thm:twoloop}]
Theorem \ref{prop:sharp} shows the primal-dual problem \eqref{eq:problem_primal_dual} is a sharp problem on any bounded region. Then, applying Theorem 1 and Theorem 2 in \cite{applegate21b}, we finish the proof of Theorem \ref{thm:twoloop}.
\end{proof}



\subsection{Profit Estimation and Prediction}\label{subsec:targetingvariables}

There exist many models that we can use to achieve the profit prediction. Comparing the differences between different models is not the goal of our paper, and we choose to use LASSO with a full set of interactions to predict the profit given recommendations from literature (e.g., \citealt{athey17} and \citealt{prospective20}). We use all available training data in this step ($N=2,455,727$). The estimation function we consider is: 
\begin{equation}
    P_{i}^{j} = f(W_i^j,O_i) + \epsilon_i =  \alpha + \beta W_{i}^j + \gamma O_{i} + \delta W_{i}^j O_{i} + \epsilon_{i}.
    \label{eq:lasso}
\end{equation}
Here, $P_i^j$ denotes the profit earned from customer $i$ if it receives marketing action $j = 0,1,...,5$, $W_i^j$ denotes the marketing action $j = 0,1,...,5$ treated on customer $i$, $O_i$ denotes all of the covariates (contextual variables) of customer $i$. $j=0$ indicates the no-action control condition. $W_i^jO_i$ is a full interaction of treatment $W_i^j$ with all covariates $O_i$. LASSO can help to select which covariates are important. Once we have an estimated model $\hat{f}(W_i^j,O_i), j=0, 1,...,5$, we can then derive predicted profit for each customer under each marketing action $\hat{P}_i^j = \hat{f}(W_i^j = 1, W_i^{j^{\prime}}=0 ,O_i)$ for $j = 0,1,...,5$ and $j^{\prime} \neq j$. In Problem (\ref{eq:problem_setup}), $p_i^j$ is the incremental profit that the firm earns from customer $i$ if it receives marketing action $j=1,2,3,4,5$. The benchmark condition is $j=0$. Thus, we derive $p_i^j$ by calculating $\hat{P}_i^j - \hat{P}_i^0$ for $j = 1,2,3,4,5$. We can also consider directly estimating the incremental profit in Equation (\ref{eq:lasso}). We simplify the estimation and prediction step to focus on the optimization step for targeting with constraints problem. In Section \ref{sec:conclusion}, we discuss several interesting directions to think about the interaction between prediction and optimization. 


Table \ref{table:summary} reports the summary statistics for all of the targeting variables ($O_i$) we use in the estimation and prediction model. The meanings of all variables are as follows. 
\onehalfspacing
\begin{itemize}
    \item \textit{Single Family} and \textit{Multi Family} are binary flags indicating whether the household's home is a single or multi-family home. 
    \item \textit{Member Tier} is a tier assigned to each customer by the retailer. Lower tiers indicate higher potential values. There are 10 tiers in total, and we use binary flags for the first 9 tiers in the profit prediction. 
    \item \textit{Child} is a binary flag indicating whether the household includes one or more children. 
    \item \textit{Female} and \textit{Male} are binary flags indicating whether the head of the household is female or male. There are households for which we do not observe the gender of the household head. 
    \item \textit{Home Value Tier} is a tier classification of the household's estimated home value.  Higher tiers indicate higher values and households that do not have estimated home values are in tier 11.  There are 11 tiers in total, and we use binary flags for the first 10 tiers in the profit prediction. 
    \item \textit{Family Number} is the number of people living in the household. 
    \item \textit{Length of Residence} is the length of time living in the current home. 
    \item \textit{Income} is the estimated household income. 
    \item \textit{Age} is the age of the head of the household. 
    \item \textit{Age Type} is a binary flag indicating whether the age is estimated. 
    \item \textit{Homeowner}, \textit{Renter} and \textit{Condo Owner} are each binary flags indicating whether the household is a homeowner, renter or condo owner. 
    \item \textit{Residential}, \textit{Condominium}, \textit{Duplex}, \textit{Apartment}, \textit{ Agricultural} and \textit{Mobile Homes} are binary flags indicating the property type. 
    \item \textit{Distance} is the distance from the household to the retailer's nearest store. 
    \item \textit{Comp. Distance} is the distance from the household to the nearest competitor's store. 
    \item \textit{3yr Response} is the average response rate over the last 3 years to the retailer's prospecting marketing activities (in that carrier route). 
    \item \textit{Penetration Rate} is the percentage of households in that five-digit zip code that are already members of the retailer. 
    \item \textit{F Flag} is a binary flag indicating whether the retailer considers the household's ZIP code as ``far'' from the retailer's nearest store. 
    \item \textit{M Flag} is a binary flag indicating whether the retailer considers the household's ZIP code to be a ``medium distance'' from the retailer's nearest store.
\end{itemize}
\doublespacing

\onehalfspacing
\begin{table}[H]
\centering
\begin{threeparttable}
\caption{Summary Statistics of Targeting Variables.} \vspace{.2in}
 \small
\begin{tabular}{ l c c c }
 \hline\hline
 Variable & Mean & Median & Standard Deviation \\
 \hline

 Single Family & 0.808 & 1.000 & 0.394 \\
 Multi Family & 0.188 & 0.000 & 0.391\\
 Member Tier & 5.365 &  5.000 & 2.476\\
 Child & 0.231 & 0.000 & 0.422\\
 Female & 0.317 & 0.000 & 0.465\\
 Male & 0.658 & 1.000 & 0.474 \\
 Home Value Tier& 4.071 & 2.000 & 3.506 \\
 Family Number & 2.491 & 2.000 & 1.691\\
 Length of Residence & 11.94 & 8.000 & 11.68\\
 Income (in 1,000s) & 63.83 & 50.00 & 52.46\\
 Age & 50.03 & 50.00 & 16.96\\
 Age Type & 0.824 & 1.000 & 0.381\\
 Homeowner & 0.663 & 1.000 & 0.473\\
 Renter & 0.243 & 0.000 & 0.429\\
 Condo Owner & 0.033 & 0.000 & 0.179 \\
 Residential & 0.682 & 1.000 & 0.466 \\
 Condominium & 0.305 & 0.000 & 0.172\\
 Duplex & 0.025 & 0.000 & 0.157 \\
 Apartment & 0.002 & 0.000 & 0.047 \\
 Agricultural & 0.009 & 0.000 & 0.096\\
 Mobile Homes & 0.025 & 0.000 & 0.155 \\
 Distance & 10.84 & 8.095 & 8.183 \\
 Comp. Distance & 9.240 & 6.433 & 7.664\\
 3yr Response & 0.121 & 0.994 & 0.205 \\
 Penetration Rate & 0.093 & 0.069 & 0.071\\
 F Flag & 0.590 & 1.000 & 0.492\\
 M Flag & 0.280 & 0.000 & 0.449\\
 \hline\hline
\end{tabular}
\label{table:summary}
\end{threeparttable}
\end{table}
\doublespacing

\newpage

\subsection{Varying the Number of Customers: Alternative Hardware Options}\label{subsec:numberhouseholds}
\begin{table}[h]
\onehalfspacing
\centering
\begin{threeparttable}
\caption{Varying the Number of Customers: Alternative Hardware Options}\vspace{.2in}
\begin{tabular}{ l c c c c} 
 \hline\hline
 & Primal Simplex & Dual Simplex  & Barrier & Our Algorithm \\
 \hline
 \textbf{$H_1$ Hardware Specification}\\
  Full Sample ($I=2,065,758$) & - & - & - & -\\
 75\% Sample ($I=1,549,323$) &   - & - & - & - \\
  50\% Sample ($I=1,032,884$)  &  - & - & - & - \\ 
  25\% Sample ($I=516,435$)  & * & * & * & 10,780\\
 \hline
 \textbf{$H_2$ Hardware Specification}\\
  Full Sample ($I=2,065,758$) & - & - & - & -\\
 75\% Sample ($I=1,549,323$) &   - & - & - & 275,000 \\
  50\% Sample ($I=1,032,884$)  &  * & * & * & 16,870 \\ 
  25\% Sample ($I=516,435$)  & 47,956 & 10,318 & * & 5,345\\
 \hline\hline
\end{tabular}
\begin{tablenotes}
\footnotesize
\item \singlespacing \vspace{-.1in} Notes. The table reports the total computation time (in seconds) used by each method to solve each instance of the IPwC problem, when varying both (a) the proportion of customers included in each problem, and (b) the hardware specifications. If the method cannot solve an instance of the problem, the table reports either ``-'' indicating it ran out of memory, or ``*'' indicating it ran out of time. 
\end{tablenotes}
\label{table:varyIH1}
\end{threeparttable}
\end{table}
\doublespacing

\newpage
\subsection{Varying the Number of Segments: Alternative Hardware Options}\label{subsec:numbersegments}
\begin{table}[h]
\onehalfspacing
\centering
\begin{threeparttable}
\caption{Varying the Number of Segments: Alternative Hardware Options}\vspace{.2in}
\begin{tabular}{ l c c c c} 
 \hline\hline
 & Primal Simplex & Dual Simplex  & Barrier & Our Algorithm \\
 \hline
 \textbf{$H_1$ Hardware Specification}\\
  Five-digit zip codes ($K=229$) & - & - & - & - \\
  Four-digit zip codes ($K=87$) &   - & - & - & - \\
  Three-digit zip codes ($K=18$)   &  94,500 & 1,567 & 13,464 & 1,394 \\ 
 \hline
 \textbf{$H_2$ Hardware Specification}\\
  Five-digit zip codes ($K=229$) & - & - & - & - \\
  Four-digit zip codes ($K=87$) &   * & * & * & 13,870 \\
  Three-digit zip codes ($K=18$)   &  84,016 & 1,477 & 12,509 & 777 \\ 
 \hline\hline
\end{tabular}
\begin{tablenotes}
\footnotesize
\item \singlespacing \vspace{-.1in} Notes. The table reports the total computation time (in seconds) used by each method to solve each instance of the IPwC problem, when varying both (a) the coarseness of the customer segmentation, and (b) the hardware specifications. If the method cannot solve an instance of the problem, the table reports either ``-'' indicating it ran out of memory, or ``*'' indicating it ran out of time. 
\end{tablenotes}
\label{table:varyIH1}
\end{threeparttable}
\end{table}
\doublespacing

\subsection{Parameters Used for the Constraints Implemented in Section \ref{sec:validation}}\label{app:parameters}

\onehalfspacing
\noindent \textbf{Volume I} 

\noindent The firm imposes requirements on the minimum and maximum number of customers in customer $k$ for marketing action $j$. Specifically, we set these requirements as: $a_k^1 = 0.3n_k$, $b_k^1 = 0.35n_k$, $a_k^2 = 0.05n_k$, $b_k^1 = 0.1n_k$, $a_k^3 = 0.05n_k$, $b_k^3 = 0.1n_k$, $a_k^4 = 0.3n_k$, $b_k^4 = 0.35n_k$, $a_k^5 = 0.05n_k$, and $b_k^5 = 0.1n_k$ for $k = 1,...,K$.

\medskip
\noindent \textbf{Volume II} 

\noindent The firm asks for a performance constraint that at least 70\% of the customers in segment $k$ receive the experimental marketing actions. This means that $L_k = 0.7n_k$, $U_k = \infty$, and $c_i^j=1$ for $i=1 ,...,I$, $j=1,...,5$ and $k=1,...,K$.

\medskip
\noindent \textbf{Similarity I} 

\noindent Under this constraint, if two customer segments are geographically closer, the proportion of customers that receive marketing action $j$ are more similar. In particular, when we define customer segments using five-digit zip codes, we set $g_j^{k_1k_2}=0$ for all $j=1,...,5$, $k_1=1,...,K, k_2=k_1+1,...,K$. We design $\lambda_j^{k_1k_2}$ using the following structure:
    \begin{itemize}
        \item When two segments have the same four-digit zip codes, $\lambda_j^{k_1k_2}=1.1$; 
        \item If two segments have different four-digit zip codes but the same three-digit zip codes, $\lambda_j^{k_1k_2}=1.2$; 
        \item If two segments have different three-digit zip codes, but are in the same state, $\lambda_j^{k_1k_2}=1.3$;
        \item If two segments are in different states, $\lambda_j^{k_1k_2}=1.4$. 
    \end{itemize}
    When we define customer segments using four-digit or three-digit zip codes, we apply the same definition for $\lambda_j^{k_1k_2}$ (if applicable). For example, if we define customer segments using four-digit zip codes, we start by asking whether the two customer segments have the same three-digit zip codes? If yes, then $\lambda_j^{k_1k_2}=1.2$.
    
\medskip
\noindent \textbf{Similarity II} 

\noindent For $\gamma^{k_1k_2}$ and $h^{k_1k_2}$, we use the same structure as we use for $\lambda_j^{k_1k_2}$ and $g_j^{k_1k_2}$. For $d_i^j$, we set $d_i^1 = 0.3$, $d_i^2 = 0.3$, $d_i^3 = 0.2$, $d_i^4 = 0.1$, and $d_i^5 = 0.1$ for $i = 1,..., I$.

\medskip
\noindent \textbf{Targeting} 

\noindent For all customers $i=1,...,I$, we set $M_i=1$, which means that each customer $i$ can receive at most one marketing action. 

\doublespacing



\end{document}